\documentclass[a4paper,11pt]{amsart}

\usepackage[british]{babel}
\swapnumbers

\numberwithin{equation}{subsection}
\numberwithin{figure}{subsection}
\numberwithin{table}{subsection}

\usepackage{amssymb}
\usepackage{mathtools} 
\usepackage{amstext}
\usepackage{amsthm}

\usepackage{lmodern}
\usepackage{mathrsfs}

\usepackage{fancyhdr}
\usepackage{xcolor}
\usepackage[T1]{fontenc}

\usepackage{stmaryrd}
\usepackage{textcomp}

\usepackage{url}
\usepackage{hyperref}

\usepackage[cmtip,arrow,matrix,curve,tips,frame]{xy}
\xyoption{matrix}

\usepackage{tikz-cd}

\theoremstyle{plain}
\newtheorem{thm}[subsection]{Theorem}
\newtheorem*{thm*}{Theorem}

\newtheorem{lem}[subsection]{Lemma}
\newtheorem*{lem*}{Lemma}

\newtheorem{prop}[subsection]{Proposition}
\newtheorem*{prop*}{Proposition}

\newtheorem{cor}[subsection]{Corollary}
\newtheorem*{cor*}{Corollary}

\newtheorem*{conj*}{Conjecture}

\theoremstyle{definition}

\newtheorem{defn}[subsection]{Definition}
\newtheorem*{defn*}{Definition}

\newtheorem{notn}[subsection]{Notation}
\newtheorem*{notn*}{Notation}

\theoremstyle{remark}
\newtheorem{rmk}[subsection]{Remark}

\newtheorem*{rmk*}{Remark}

\newtheorem*{exa*}{Example}

\newtheorem*{claim*}{Claim}
\newtheorem{claimsub}[equation]{Claim}

\newcounter{listnum}

\newcounter{asslistcounter}
\newenvironment{assertionlist}{
 \begin{list}
  {\upshape (\alph{asslistcounter})}
  {\setlength{\leftmargin}{18pt}
   \setlength{\rightmargin}{0pt}
   \setlength{\itemindent}{0pt}
   \setlength{\labelsep}{5pt}
   \setlength{\labelwidth}{13pt}
   \setlength{\listparindent}{\parindent}
   \setlength{\parsep}{0pt}
   \setlength{\itemsep}{0pt}
   \setlength{\topsep}{-.5\parskip}
   \usecounter{asslistcounter}}}
  {\end{list}}

\newcounter{subenvcounter}
\newenvironment{subenv}{%
 \begin{list}
  {\em (\arabic{subenvcounter})}
  {\setlength{\leftmargin}{20pt}
   \setlength{\rightmargin}{0pt}
   \setlength{\itemindent}{0pt}
   \setlength{\labelsep}{5pt}
   \setlength{\labelwidth}{13pt}
   \setlength{\listparindent}{\parindent}
   \setlength{\parsep}{0pt}
   \setlength{\itemsep}{0pt}
   \setlength{\topsep}{-\parskip}
   \usecounter{subenvcounter}}}
  {\end{list}}

\newcommand{\restr}[2]{{#1}\raise-.5ex\hbox{\ensuremath|}_{#2}}

\newcommand{\cl}[1]{\mkern 1.5mu\overline{\mkern-1.5mu#1\mkern-1.5mu}\mkern 1.5mu}

\def \Lra   {\Leftrightarrow}

\def \mono  {\hookrightarrow}

\def \epi   {\twoheadrightarrow}

\def \isom  {\stackrel{\sim}{\rightarrow}}

\def \bij   {\stackrel{1:1}{\rightarrow}}

\def \iv   {^{-1}}

\DeclareMathOperator{\Ad}{Ad}

\DeclareMathOperator{\Aut}{Aut}

\DeclareMathOperator{\codim}{codim}
\DeclareMathOperator{\cha}{char}

\DeclareMathOperator{\defect}{def}
\DeclareMathOperator{\diag}{diag}

\DeclareMathOperator{\Def}{Def}

\DeclareMathOperator{\image}{im}

\DeclareMathOperator{\Gal}{Gal}
\DeclareMathOperator{\GL}{GL}

\DeclareMathOperator{\Irr}{Irr}

\DeclareMathOperator{\Lie}{Lie}

\DeclareMathOperator{\PGL}{PGL}

\newcommand{\Sh}{\mathsf{Sh}}

\DeclareMathOperator{\Spec}{Spec}
\DeclareMathOperator{\Spf}{Spf}

\DeclareMathOperator{\trace}{Tr}
\DeclareMathOperator{\Res}{Res}

\DeclareMathOperator{\rank}{rk}

\DeclareMathOperator{\val}{val}

\def \Flag {{\mathscr{F}\hspace{-.2em}\ell}}

\def \afr {\mathfrak{a}}

\def \frm {{\mathrm{f}}}

\def \fsf {{\mathsf{f}}}

\def \Gsf {{\mathsf{G}}}

\def \Xsf {{\mathsf{X}}}

\def \bbf {{\mathbf{b}}}

\def \Pbf {{\mathbf{P}}}

\def \hbar {{\overline{h}}}

\def \Lbar {{\overline{L}}}

\def \Ubar {{\overline{U}}}

\def \Xbar {{\overline{X}}}

\def \Ctilde {{\tilde{C}}}

\def \Ltilde {{\tilde{L}}}

\def \Wtilde {{\tilde{W}}}
\def \Xtilde {{\tilde{X}}}

\def \Bhat {{\hat{B}}}

\def \Ghat {{\hat{G}}}

\def \Shat {{\hat{S}}}
\def \That {{\hat{T}}}

\def \Xhat {{\hat{X}}}

\def \Acal {{\mathcal{A}}}

\def \Gcal {{\mathcal{G}}}

\def \Ascr {{\mathscr{A}}}

\def \Gscr {{\mathscr{G}}}

\def \Kscr {{\mathscr{K}}}

\def \Oscr {{\mathscr{O}}}

\def \Sscr {{\mathscr{S}}}

\def \Xscr {{\mathscr{X}}}

\def \AA {{\mathbb{A}}}
\def \BB {{\mathbb{B}}}

\def \GG {{\mathbb{G}}}

\def \NN {{\mathbb{N}}}

\def \PP {{\mathbb{P}}}
\def \QQ {{\mathbb{Q}}}

\def \ZZ {{\mathbb{Z}}}

\def \FFbar {{\overline{\mathbb{F}}}}

\def \Ebreve {{\breve{E}}}
\def \Fbreve {{\breve{F}}}

\usepackage[figuresright]{rotating}

\newcommand*{\p}{%
    \mathchoice%
        {\turnbox{12}{$\displaystyle\,'$}\;}%
        {\turnbox{12}{$\textstyle\,'$}\;}%
        {\turnbox{12}{$\scriptstyle\,'$}\;}%
        {\turnbox{12}{$\scriptscriptstyle\,'$}\;}%
}%

\def \unif {\varpi}

\def \k {{\FFbar_p}}
\def \L {{\Fbreve}}
\def \I {{\Gamma_0}}
\def \dom {{\mathrm{dom}}}
\def \perf {{\mathrm{perf}}}

\DeclareMathOperator{\type}{type}

\author[P.~Hamacher]{Paul Hamacher}
\address{Paul Hamacher\\
Technische Universit\"at M\"unchen\\
Zentrum Mathematik - M 11\\
Boltzmannstra{\ss}e 3\\
85748 Garching\\
Deutschland\\  }
\email{hamacher@ma.tum.de}

\title[ADLV for quasi-split groups]{On the geometry of affine Deligne-Lusztig varieties for quasi-split groups}

\begin{document}

 \begin{abstract}
  In this paper we discuss the geometry of affine Deligne Lusztig varieties with very special level structure, determining their dimension and connected and irreducible components. As application, we prove the Grothendieck conjecture for Shimura varieties with very special level at $p$ and a function field analogue. 
 \end{abstract}

 \maketitle

 \section{Introduction}

  We fix a nonarchimedian local field $F$ and denote by $\L$ the completion of its maximal unramified extension. We denote by $\Gamma$ and $I$ the absolute Galois group of $F$ and $\L$, respectively, and by $\sigma \in \Aut_F(\L)$ the Frobenius morphism.
  
Let $G$ be a quasi-split connected reductive group over a local field $F$. 
We fix a special $F$-torus $S \subset G$, i.e.\ $S$ is a maximal $\L$-split torus and contains a maximal $F$-split torus. Moreover, let $T = Z_G(S)$ be its centralizer, which is a maximal $F$-torus, and let $B = T\cdot U$ be a (rational) Borel subgroup.
  We denote by $I \subset K \subset G(\L)$ be the standard Iwahori and special subgroup and by $\Kscr$ the parahoric group scheme corresponding to $K$. 
  
  Let $\Flag_G$ denote the affine flag variety whose $k$-points naturally identify with $G(L)/K$. For $\mu \in X_\ast(T)_{I,\dom}$ and $b \in G(\L)$, the affine Deligne-Lusztig variety is the locally closed subset
  \[
   X_\mu(b)(k) = \{gK \in \Flag(k) \mid g\iv b\sigma(g) \in K\varpi^\mu K \}.
  \]
  We equip $X_\mu(b)$ with reduced structure, making it a scheme which is locally of finite type (if $\cha F = p$) or locally of perfectly finite type (if $\cha F = 0$). It is equipped with a $J_b(F)$-action, where $J_b$ is the linear algebraic group given by
  \[
   J_b(R) = \{g\in G(\L \otimes R) \mid g\iv b\sigma(g) = b\}. 
  \]   
  
 There has been much progress in studying the geometry of affine Deligne-Lusztig varieties in the recent years. Motivated by their applications in Shimura varieties and the Langlands program one has studied the following basic questions regarding affine Deligne Lusztig varieties with parahoric level $K$:
 \begin{subenv}
  \item What is their dimension ?
  \item What are their connected components ?
  \item What are their irreducible components ?
 \end{subenv} 
 In \cite{He:GeomADL}, He defined the virtual dimension and proved that it is an upper bound in the case that $G$ is quasi-split and splits over a tamely ramified extension. Recently Milisevic and Viehmann (\cite{MilicevicViehmann:GenericNP}), and He (\cite {He:CordialElts}) have proven that equality holds in most cases, but it is known that the formula does not hold in general.  
  
 If $b$ is HN-irreducible it has been conjectured by Zhou that the connected components are given precisely by their intersection with connected components in the affine flag variety (the general case can be deduced from that). This statement has been proven for unramified groups by Nie (\cite{Nie:ADLVIrredComp}) and for basic $b$ or residually split $G$ by He and Zhou (\cite{HeZhou:ADLVcomp}).
  
 For unramified $G$, the set $\Irr X_\mu(b)$ of topdimensional irreducible components of $X_\mu(b)$ have been shown to be in canonical bijection with the Mirkovic-Vilonen basis $\BB_\mu(\lambda)$ of a certain weight space $V_\mu(\lambda)$ of the highest weight representation of the dual group $\Ghat$ in \cite{HamacherViehmann:ADLVIrredComp}, \cite{Nie:ADLVIrredComp}, \cite{ZhouZhu:IrredCompADLV}, which was originally conjectured by Chen and Zhu. Moreover Zhou and Zhu deduced the numerical version of the conjecture, i.e.\ that both sets have the same cardinality, for any very special $K$ by a formal argument from the hyperspecial case. The variety $X_\mu(b)$ is known to equidimensional for most cases so that one may ignore the ``equidimensional'' requirement (see Corollaries~\ref{cor-equidim} and \ref{cor-equidim-2} below).
 
  For general parahoric level structure and groups the questions are still widely open. In this manuscript we therefore concentrate on the case that $K$ is a maximal parahoric subgroup. Before stating our main result, we present a slightly different viewpoint on the $X_\mu(b)$ and the $J_b(F)$-actions. Denote by $[b]$ the $\sigma$-conjugacy class of $b$ in $G(L)$, considered as reduced sub-ind-scheme in $LG$. We denote by $\Gcal_{[b], \mu}$ the ``universal local $\Kscr$-shtuka over $K\unif^\mu K \cap [b]''$, i.e. whose fibre over $b'$ equals $(L^+\Kscr,b'\sigma)$. Then the local model map 
  \begin{align*}
   \ell_b\colon \Xtilde_\mu(b) \coloneqq \{g \in LG \mid g\iv b\sigma(g) \in K\unif^\mu K\} &\to [b] \cap K\unif^\mu K  \\
   g &\mapsto g\iv b\sigma(g)
  \end{align*}
  canonically identifies $\Xtilde_\mu(b)$ with the moduli space of quasi-isogenies \newline 
   $(L^+\Kscr,b\sigma) \to \Gcal_{[b], \mu}$. By construction $J_b(F)$ acts simply transitively on the gemetric fibres (in fact the morphism is a pro\'etale $\underline{J_b(F)}$-torsor up to universal homeomorphism).
  \begin{center}
   \begin{tikzcd}
    & \Xtilde_\mu(b) \arrow[two heads]{dl}[swap]{L^+\Kscr} \arrow[two heads]{dr}{\underline{J_b(F)}} & \\
    X_\mu(b) & & \left[b\right] \cap K \unif^\mu K.
   \end{tikzcd}
  \end{center}
  In particular, $J_b(F)$-orbits of connected/irreducible components of $X_\mu(b)$ are canonically identified with the respective components of $[b] \cap K\unif^\mu K$.
  We denote the defect of $b$ by $\defect(b) \coloneqq \rank G - \rank J_b$. 
  
  \begin{thm} \label{thm-main}
   \begin{subenv}
    \item The dimension of $X_\mu(b)$ equals the virtual dimension \[ d_G(\mu,b) \coloneqq \langle \rho, \mu-\nu\rangle - \frac{1}{2}\defect (b),\] where $\defect(b) \coloneqq \rank_F(G) - \rank_F(J_b)$ equals the rank of defect of $b$.
    \item There exists a canonical bijection $J_b(F) \backslash \Irr X_\mu(b) \bij \BB_\mu(\lambda([b]))$ where $\lambda([b]) \in X_\ast(S)$ denotes the best integral approximation of the Newton point (see \S~\ref{ssect-invariants} for details).
    \item If $(b,\mu)$ is HN-irreducible, then the connected components of $X_\mu(b)$ are precisely the intersection of $X_\mu(b)$ with the connected components of $\Flag_G$.
   \end{subenv}
  \end{thm}

  We can give a more explicit description of bijection in (2).  Assume that $b$ is of standard form and denote by $M$ the (unique) Levi subgroup such that $[b] \cap M(L)$ is superbasic and $P = M N$ the corresponding standard parabolic subgroup of $G$. For $C \in \Irr X_\mu(b)$ denote by $\lambda_C \in X_\ast(T)_I$ the cocharacter such that $C$ is generically contained in $N(L) I_M \varpi^{\lambda_C} K$ and let $w_C$ the shortest element in the Weyl group of $M$ such that $w_C\iv (\lambda_C)$ is $M$-dominant. Moreover we denote by $\Ctilde \subset LG$ the preimage of $C$.
  
 \begin{prop} \label{prop-ic}
 Denote by $Z_C \subset U\unif^{\tilde\lambda}K \cap K \mu K$ (for a certain $\tilde\lambda\in X_\ast(T)_I$ restricting to $\lambda([b])$) be the Mirkovic-Vilonen cycle corresponding to the image of $C$ in $\BB_\mu(\lambda([b]))$. Then $Z_C$ is determined by the following properties.
  \begin{subenv}
   \item $\tilde\lambda = b\sigma(\lambda) - \lambda$.
   \item The set $\ell_b(\Ctilde \cap N(L) I_M \unif^{\lambda_C}) \cdot K$ is an open dense subset of $w_C(Z_C)$.  
  \end{subenv}
 \end{prop}  
  
 As a consequence of the results above, we obtain the following application to  Shimura varieties. Let $(\Gsf, \Xsf)$ be a Shimura datum of Hodge-type such that $\Gsf_{\QQ_p}$ is quasi-split, $K \subset G(\AA_f)$ a small enough compact open subgroup whose level $K$ at $p$ is as above. We denote by $\mu$ the dominant representative in the conjugacy class of cocharacters of $\Gsf$ associated to $\Xsf$. 
 Kisin and Pappas constructed an integral model $\Sscr_{\Gsf}$ of $\Sh(\Gsf,\Xsf)_X$ under some tame additional technical condition. It is equipped with a polarised Abelian scheme  $\Ascr \to \Sscr_\Gsf$ as well as tensors on the display of $\Ascr[p^\infty]$ over the completion of the special fibre of $\Sscr_\Gsf$. In particular, we obtain a $\Gsf_{\QQ_p}$-isocrystal over the perfection of the special fibre $\Sscr_{\Gsf,0}$, defining a Newton stratification on $\Sscr_{\Gsf,0}$. We denote the Newton stratum corresponding to a $\sigma$-conjugacy class $[b]$ by $\Sscr_{\Gsf,0}^{[b]}$. 
 
\begin{thm}[cf.~Cor.~\ref{cor-Newton-Sh}]
  For any $[b] \in B(\Gsf_{\QQ_p},\mu)$, we have
  \begin{subenv}
   \item $\Sscr_{\Gsf,K,0}^{[b]}$ is of pure dimension $\langle \rho, \mu + \nu([b]) \rangle - \frac{1}{2} \defect_G(b)$. 
   \item $\cl{\Sscr_{\Gsf,K,0}^{[b]}} = \bigcup_{[b'] \leq [b]} \Sscr_{\Gsf,K,0}^{[b']}$.
  \end{subenv}
 \end{thm}

 The proof of the main results is a generalisation of their respective proofs in the unramified case. In sections~\ref{sect-group-theory} and \ref{sect-ADLV} we recall the necessary structure theory of reductive group and facts about affine Deligne-Lusztig varieties and generalise them to the ramified case if necessary. We prove the statements about affine Deligne-Luztig varieties in the basic case by first reducing to the case of minuscule $\mu$ using a technique in \cite{Nie:ADLVIrredComp} and then calculate the dimension and number of irreducible components by considering a cellular decomposition of $X_\mu(b)$ parametrised by combinatorial objects called EL-charts as in \cite{HamacherViehmann:ADLVIrredComp}. In section~\ref{sect-reduction-step} we prove that the superbasic case implies the general case by generalising the calculations made for split groups in \cite{GHKR:DimADL}. In section~\ref{sect-cc} we  explain, how to generalise the Nie's proof for the description of connected components for unramified groups in \cite{Nie:ADLVcomp} to our situation.  As a consequence of the geometric properties properties of affine Deligne-Lusztig varieties proven above, we deduce the result on the geometry of Newton strata in Shimura varieties as well as equidimensionality in equal characteristic as well as for the most interesting cases in mixed characteristic in section~\ref{sect-Shimura}.
 
 {\bf Acknowledgements:} I am grateful to Tom Haines for his explanations about the Kato-Lusztig formula in the ramified case. Many thanks go to Timo Richarz for his explanations about the Satake equivalence. The author was partially supported by ERC Consolidator Grant 770936: NewtonStrat. 
 
\section{Group theoretical background} \label{sect-group-theory}

\subsection{} Let $G$ be a quasi-split connected reductive group over a local field $F$. 
We fix a special $F$-torus $S \subset G$, i.e.\ $S$ is a maximal $L$-split torus and contains a maximal $F$-split torus. Moreover, let $T = Z_G(S)$ be its centralizer, which is a maximal $F$-torus, and let $B = T\cdot U$ be a (rational) Borel subgroup.
We denote by $\sigma$ the Frobenius automorphism in $\Gal(\Fbreve/F)$ as well as the induced automorphism on $G(\Fbreve)$ by $\sigma$.

We denote by $\Acal = X_\ast(S)_\QQ$ the apartment of the Bruhat-Tits building of $G$ corresponding to $S$. Let $\afr$ be the alcove in the dominant Weyl chamber of $\Acal$ whose closure contains $0$. We denote by $I \subset G(\L)$ the Iwahori subgroup
stabilizing $\afr$  and by $K \subset G(L)$ the special parahoric subgroup stabilizing $0 \in \Acal$.

\subsection{} Let $N$ be the normalizer of $T$. By definition, the finite Weyl group
associated with $S$ is
\[
 W = N(L)/T (L)
\]
and the Iwahori-Weyl group associated with S is
\[
 \Wtilde = N(L)/T(L)_1
\]
Here $T(L)_1$ denotes the unique parahoric subgroup of $T(L)$.

Let $I$ be the absolute Galois group of $\L$. By \cite[Appendix]{PappasRapoport:AffineFlag}, we may identify $T(\L)/T(\L)_1$ with the Galois coinvariants $X_\ast(T)_I$. We denote by $\unif^\lambda$ the element of $T(\L)/T(\L)_1$ corresponding to $\lambda \in X_\ast(T)_I$. Moreover, we obtain the following short exact sequence:
\begin{center}
 \begin{tikzcd}
  1 \ar{r}& X_\ast(T)_I \ar{r} & \Wtilde \ar{r} & W  \ar{r} & 1
 \end{tikzcd}
\end{center}
The stabiliser of $0 \in \Acal$ maps isomorphically onto $W$. Thus the short exact sequence splits and we obtain a semidirect product $\Wtilde = X_\ast(T)_I \rtimes W$. 

We have the Iwahori-Bruhat decompostion
\[
 G(\Fbreve) = \bigsqcup_{x \in \Wtilde} IxI,
\]
which induces $K = \bigsqcup_{w \in W} IwI$. Thus we also obtain
\[
 G(\Fbreve) = \bigsqcup_{\lambda \in X_\ast(T)_I} I\varpi^\lambda K
 = \bigsqcup_{\mu \in X_\ast(T)_{I,\rm dom}} K\varpi^\mu K.
\]

\subsection{}
Let $\Phi$ be the set of (relative) roots of $G$ over $L$ with respect to $S$ and $\Phi^+ \subset \Phi$ the set of positive roots corresponding to $B$. For each $\alpha \in \Phi$ we denote by $U_\alpha \subset G$ its root subgroup. By Bruhat-Tits theory there exists a canonical filtration $(U_{\alpha,m})_{m\in\QQ}$ of $U(\Fbreve)$; we denote $U_{\alpha,m+} \coloneqq \bigcap_{\epsilon > 0} U_{\alpha,m+\epsilon}$ (For details see e.g. \cite[\S~4]{Landvogt:Compactification}). The set of affine roots $\Phi_a$ is defined as the functions $\alpha + m$ such that $U_{\alpha,m+} \not= U_{\alpha,m}$ if $2\alpha\not\in \Phi$ and $U_{\alpha,m+} \not\subseteq U_{\alpha,m}U_{2\alpha}$ otherwise. 

There exists a reduced root system $\Sigma$ on $\Acal^\ast$ such that the set of  hyperplanes defined by the associated affine root system $\Sigma_a \coloneqq \Sigma+\ZZ$ equals the set of hyperplanes defined by the elements of $\Phi_a$. The roots of $\Sigma$ are proportional to the roots of $\Phi$. Let $\alpha'+m' \in \Sigma$ and let $\alpha+m \in \Phi_a$ the associated affine root. If $\alpha/2 \not\in \Phi$, we denote
\[
 U'_{\alpha',m'} =
 \begin{cases}
  U_{\alpha,m} & \textnormal{if } 2\alpha, \not\in \Phi \\
  U_{\alpha,m-1} \cdot U_{2\alpha,2m} & \textnormal{if } 2\alpha, \in \Phi.
 \end{cases} 
\]
See also \cite[\S~9.a,9.b]{PappasRapoport:AffineFlag} for an explicit description. For any $\alpha \in \Sigma, m\in \ZZ$ the group $U'_{\alpha,m}$ is stable under conjugation with $T(\L)_1$ and satisfies $\unif^\lambda U'_{\alpha',m} \unif^{-\lambda} = U'_{\alpha',m+\langle \alpha,\lambda\rangle}$ for any $\lambda\in X_\ast(T)_I$.
By Iwahori decomposition we have
\[
 I = T(L)_1 \cdot \prod_{\alpha \in \Sigma^+} U'_{\alpha,0} \cdot \prod_{\alpha \in \Sigma^-} U'_{\alpha,1},
\]
where all products are direct products.

\begin{defn}
 Let $\lambda \in X_\ast(T)_I$. For any root $\alpha \in \Phi$, we write $\alpha >_\lambda 0$ if either $\langle \alpha,\lambda \rangle > 0$ or $\alpha \in \Phi^+$ and $\langle \alpha, \lambda \rangle = 0$. We write $\alpha <_\lambda 0$ otherwise. Moreover we denote
 \[
  U_\lambda \coloneqq \prod_{\alpha >_\lambda 0} U_\alpha \textnormal{ and } \Ubar_\lambda \coloneqq \prod_{\alpha <_\lambda 0} U_\alpha.
 \]
\end{defn}
 Note that $U_\lambda$ and $\Ubar_\lambda$ are opposite maximal unipotent subgroups of $G$. More precisely $U_\lambda = w_\lambda U w_\lambda\iv$, where $w_\lambda \in W$ denotes the shortest element such that $w_\lambda\iv (\lambda)$ is dominant. 
 
 Furthermore we denote $I_\lambda \coloneqq \unif^{-\lambda} I \unif^\lambda$ for $\lambda \in X_\ast(T)_I$. Consider the decomposition
 \[
  I_\lambda = T(L)_1 \cdot \prod_{\alpha \in \Sigma} U'_{\alpha,\lambda_\alpha}
 \]
 with $\lambda_\alpha = -\langle \alpha, \lambda \rangle$ if $\alpha > 0$ and $\lambda_\alpha = 1-\langle \alpha, \lambda \rangle$ otherwise. Note that $\lambda_\alpha \leq 0 $ iff $\alpha >_\lambda 0$ (or equivalently $\lambda_\alpha \geq 1$ iff $\alpha <_\lambda 0$). We fix the notation $\alpha \gg_\lambda 0$ for $\lambda_\alpha < 0$ for later use.
 
 Lastly, we denote 
 \begin{align*}
  I^\lambda &\coloneqq I \cap \varpi^\lambda K \varpi^{-\lambda} = \unif^\lambda (I_\lambda \cap K ) \unif^{-\lambda} 
 \end{align*}
 Note that the canonical morphism $\Theta_\lambda\colon I \epi I\unif^\lambda K/K$ factors through an isomorphism $I/I_\lambda \isom I\unif^\lambda K/K$. 
 
 \subsection{}
 We denote by $\Flag_G$ the base change to $\FFbar_p$ of the affine flag variety (over $k_F$) associated with $G$ as in \cite[\S~1.c]{PappasRapoport:AffineFlag} and \cite[Def.~9.4]{BhattScholze:AffGr}. In particular, $\Flag_G$
is a sheaf on the fpqc-site of $\FFbar_p$-schemes ($\cha F$ = $p$) resp. of perfect $\FFbar_p$-schemes
($\cha F = 0$) with
\[
 \Flag_G (\FFbar_p) = G(L)/K
\]
which is representable by an inductive limit of finite type schemes ($\cha F = p$)
resp. of perfectly of finite type schemes ($\cha F = 0$). More generally we denote for by $K_n$ the $n$-th congruence subgroup of $K$ and by $\Flag_{G,n}$ the analogous ind-scheme with
\[
 \Flag_{G,n} (\FFbar_p) = G(L)/K_n
\]
Since the underlying topological spaces are Jacobian, we will refer to locally closed subsets by their closed $\FFbar_p$-points.

The Kottwitz map $\kappa = \kappa_G\colon G(\L) \to \pi_1(G)_I$, which maps $Kt^\mu K$ to the image of $\mu$ in $\pi_1(G)_I$, induces a bijection $\pi_0(\Flag_{G,n}) \cong \pi_1(G)_I$. For any subset $X$ of $G(L)$ or $G(L)/K_n$ and $\omega \in \pi_1(G)_I$, we denote $X^\omega \coloneqq X \cap \kappa^{-1}(\omega)$.

Since subsets of $G(\L)$ are generally easier to work with than subsets of $G(\L)/K_n$, we introduce the following notation setting the framework to study subvarieties of $\Flag_{G,n}$ via subsets of $G(\L)$.

\begin{defn}
 Let $X \subset G(\L)$ be a subset. 
 \begin{subenv} 
 \item We say that $X$ is admissible, if there exists an $n\in N$ such that it is the preimage of a locally closed subscheme $\Xbar \subset G(\L)/K_n$. In this case we define the dimension of $X$ by
 \[
  \dim_K X = \dim \Xbar - \dim K/K_n. 
 \]
 Similarly, for any property $\Pbf$ that is invariant under torsors of connected linear algebraic groups (e.g.\ irreducible, smooth) we say that $X$ has property $\Pbf$ if $\Xbar$ does.
 \item We say that $X$ is ind-admissible, if for every bounded admissible closed subset $Z \subset G(L)$ the intersection $X \cap Z$ is admissible. In this case, we define the dimension of $X$ by
 \[
  \dim_K X = \sup_{Z \subset G(L) \textnormal{ bdd.\ adm.}} \dim_K X \cap Z. 
 \]
 \end{subenv}
\end{defn}

\begin{rmk}
 Equivalently, we could define admissible sets as the preimage of any locally closed subset of some flag variety $\Flag_\Gscr$, where $\Gscr$ is an integral model $G$. However as the $K_n$ form a cofinal system, these two definitions are equivalent. In the following we also use other groups that $K_n$, most notably the $n$-the congruence group $I_n$ of $I$.
\end{rmk}

\begin{notn}
 Since any locally closed subscheme $\Xscr \subset \Flag_n(\k)$ is uniquely determined by its $\k$-valued points, we simplify the notation by referring to $\Xscr$ by $\Xscr(\k) \subset G(L)/K$. Similarly we refer to its preimage of $\tilde\Xscr \subset LG$ by its corresponding admissible subset  $\tilde\Xscr(\k) \subset G(\L)$. Moreover, if $X \subset G(L)/K_n$, we denote by $\Xtilde \subset G(\L)$ its preimage.
\end{notn}

\begin{lem} \label{lemma-admissible}
 Let $X \subset G(\L)$ be an admissible subset. Then $\sigma^a(X),gX$ and $Xg$ are admissible of the same dimension for any $a \in \ZZ, g \in G(\L)$. 
\end{lem}
\begin{proof}
 Since the maps $x \mapsto \sigma^{a}(x), x \mapsto g\cdot x$ induce isomorphisms $\Flag_{G,n} \isom \Flag_{G,n}$ the claim is obviously true for $\sigma^{a}(X)$ and $gX$. By the Cartan decomposition it suffices to show the last statement for $g\in K$ and $g = \unif^\lambda$ for $\lambda \in X_\ast(T)$. The first case follows by the same argument as above. Now $x \mapsto x\cdot\unif^\lambda$ induces an isomorphism $G(L)/I_n \isom G(L)/I_{\lambda,n}$. An easy calculation on root groups shows that $\dim I_n/(I_n \cap I_{\lambda,n}) = \dim I_{\lambda,n}/(I_n \cap I_{\lambda,n})$ and thus $\dim_K I_n = \dim_K I_{\lambda,n}$, finishing the proof.
\end{proof}

\begin{lem} \label{lemma-adm-inv}
 Let $X \subset G(\L)$ be a bounded admissible subset. Then $X\iv \coloneqq \{g \in G(L) \mid g\iv \in X\}$ is admissible and of the same dimension as $X$.
\end{lem}
\begin{proof}
 Since $X$ is bounded, it can be covered by finitely many $I$-double cosets. Thus we may assume that $X \subset IxI$ for some $x \in \Wtilde$. Now choose $n \in \NN$ such that $X$ is right $I_n$-stable and consider the commutative diagram
 \begin{center}
  \begin{tikzcd}
   IxI \arrow{r}{(\cdot)\iv} \arrow[two heads]{d} & Ix\iv I \arrow[equal]{r}{} \arrow[two heads]{d} & Ix\iv I \arrow[two heads]{d} \\
   IxI / I_n \arrow{r}{\sim} & I_n \backslash Ix\iv I \arrow{r}{\sim} & Ix\iv I / x I_n x\iv.
  \end{tikzcd}
 \end{center}
 We see that $X\iv$ is admissible and that $\dim_K i(X) = \dim X/I_n - \dim_K x I_n x\iv$. Since $\dim_K x I_n x\iv = \dim_K I_n$ by the previous lemma, the claim follows.
\end{proof}

\subsection{} \label{ssect-satake}

Recall that Satake equivalence (\cite{Zhu:RamifiedSatake},\cite{Richarz:AffGr}) states that the category of $K$-equivariant perverse sheaves on $\Flag_G$ is equivalent to the category of representations of a reductive group $G^\vee$, constructed as follows (see \cite[\S~5]{Haines:Dualities} for details). Denote by $\Ghat$ the group corresponding to the dual (absolute) root datum $(X_\ast(T),\Phi^\vee,X^\ast(T),\Phi)$. After fixing a splitting $(\Bhat,\That,\Xhat)$ of $\Ghat$, the Galois action on the root datum lifts uniquely to $\Ghat$. We define
\[
 G^\vee \coloneqq \Ghat^{I},
\]
and similarly $T^\vee \coloneqq \That^I, B^\vee \coloneqq \Bhat^I$. Then $\Ghat$ is a (possibly non connected) reductive group with splitting $(T^{\vee,0}, B^{\vee,0},\Xhat)$ whose root sytem is canonically isomorphic to $(X_\ast(T)_{I},\Sigma^\vee,X^\ast(T)^I, \Sigma)$. Moreover, the $\Gamma$-action on $\Ghat$ descends to a $\sigma$-action on $G^\vee$. We denote $S^\vee \coloneqq (T^\vee)^{\langle \sigma \rangle}$.

 For every $\mu \in X_\ast(T)_I^{\rm dom} =  X^\ast(T^\vee)_{\rm dom}$ let $V_\mu$ be the (unique) irreducible representation of $G^\vee$ with highest weight $\mu$. For every $\lambda \in X_\ast(T)_\Gamma = X^\ast(S^\vee)$, we denote by $V_\mu(\lambda)$ the corresponding weight space.  For connected groups $G^\vee$, Kashiwara and Lusztig constructed a canonical basis $\BB_\mu(\lambda)$ of $V_\mu(\lambda)$ (\cite{Kashiwara:Crystalizing},\cite{Lusztig:CanonicalBases}), another construction was given by Littelmann (\cite{Littelmann:PathsAndRootOp}). Since $\restr{V_\mu(\lambda)}{G^{\vee,0}}$ is canonically isomorphic to the $\restr{\lambda}{T^{\vee,0}}$-weight space of the highest weight module of weight $\restr{\mu}{T^{\vee,0}}$ by \cite[Lemma~5.5]{Haines:Dualities} this construction also yields a canonical basis $\BB_\mu(\lambda)$ of $V_\mu(\lambda)$ in the case when $G^\vee$ is not connected. Moreover, we denote by $\BB_\mu = \bigcup_{\lambda \leq \mu} \BB_\mu(\lambda)$ the canonical basis of $V_\mu$. 
 
\begin{prop}[cf.~{\cite[Prop.~3.3.15]{XiaoZhu:CyclesShVar}}]
 Let $\mu \in X_\ast(T)_{\rm dom}$ and $\lambda \in X_\ast(T)$. Then  $U\unif^\lambda K \cap K \mu K$ has dimension $\langle 2\rho, \mu+\lambda \rangle$ and $\dim V_\mu(\lambda)$ topdimensional irreducible components. There is exists a canonical bijection
 \[
  \Irr(U\unif^\lambda K \cap K \mu K) \bij \BB_\mu(\lambda).
 \]
\end{prop} 
\begin{proof}
 The numerical claims are proven in Proposition~\ref{prop-intersection} below. The construction of the bijection is identical to the one given in the unramified case \cite{XiaoZhu:CyclesShVar}; we briefly explain why the construction generalises. If $\mu$ is quasi-minuscule, one constructs the bijection by explicit calculations on the root groups (\cite[\S~3.2.5]{XiaoZhu:CyclesShVar}), which inhibit the same structure as in the unramified case. Next, one decomposes $\lambda = \lambda_1 + \dotsb \lambda_r$ and $\mu = \mu_1 + \dotsb + \mu_r$ with $\mu_i$ quasi-minuscule. For $\lambda_\bullet = (\lambda_1,\dotsc , \lambda_r)$ and $\mu_\bullet = (\mu_1,\dotsb, \mu_r)$ one considers the intersection of the convolution products $S_{\lambda_\bullet} = U\unif^{\lambda_1}K \tilde\times \dotsb \tilde\times U\unif^{\lambda_r}K/K$ and $\Flag_{\mu_\bullet} = K\mu_1 K \tilde\times \dotsb \tilde\times K \mu_r K/K$. From the first step one deduces as in \cite[Prop.~3.3.12]{XiaoZhu:CyclesShVar} we have a bijection between $\Irr (S_{\lambda_\bullet} \cap \Flag_{\mu_\bullet})$ and
 \[
  \{\gamma\colon [0,1] \to X_\ast(T)_{I,\QQ} \textnormal{ Littelmann path of type } \mu_\bullet \textnormal{ and } \gamma(1) = \lambda \}.
 \]
 Finally one considers the crystal of Littelman paths $\BB_\mu(\lambda)'$ (denoted $\BB_\mu(\lambda)$ in \cite{XiaoZhu:CyclesShVar}) which gives rise to the crystal basis of $V_\mu(\lambda)$. As in the proof of \cite[Prop.~3.3.15]{XiaoZhu:CyclesShVar} one shows that the irreducible components corresponding to a Littelmann path $\gamma \in \BB_\mu(\lambda)'$ under the above bijection  map onto a subset of $U\unif^\lambda K \cap K \mu K)/K$ of dimension $\langle \rho, \lambda + \mu \rangle$ and hence contains an open subset of a unique irreducible component $S_\gamma$ by the first part of the proposition. Using the arguments of the proof of \cite[Prop.~3.3.15]{XiaoZhu:CyclesShVar}, one shows that $\gamma \mapsto S_\gamma$ is injective and hence bijective by the first part of the proposition.
\end{proof}

\begin{defn}
 We denote $MV_{\mu}(\lambda) \coloneqq \Irr (U\unif^\lambda K \cap K \unif^\mu K) $ and call its elements Mirkovi\'c-Vilonen cycles.
\end{defn}

 \subsection{} \label{ssect-mv-tensor} More generally consider the canonical basis $\BB_{\mu_\bullet} \coloneqq \BB_{\mu_1} \times \dotsb \times \BB_{\mu_d}$ of $V_{\mu_\bullet} \coloneqq V_{\mu_1} \otimes \dotsc \otimes V_{\mu_d}$ for $\mu_\bullet \in X_\ast(T)^d$. We decompose $V_{\mu_\bullet} = \bigoplus V_\mu^{m_{\mu_\bullet}^\mu}$ into simple representations. Since $m_{\mu_\bullet}^{\Sigma\,\mu_i} = 1$, we obtain an embedding $V_{\Sigma\,\mu_i} \mono V_{\mu_\bullet}$, inducing $\BB_{\Sigma\,\mu_i} \mono \BB_{\mu_\bullet}$.
 
 On the other hand consider  
 $
  MV_{\mu_\bullet}(\lambda) \coloneqq \bigsqcup_{\Sigma\, \lambda_i = \lambda} MV_{\mu_i}(\lambda_i)),
 $
 the set of topdimensional irreducible components of $m_{\mu_\bullet}\iv(U\unif^\lambda K/K)$, where $m_{\mu_\bullet}\colon \Flag_{\mu_\bullet} \to \Flag$ denotes the multiplication morphism. Then the analogous construction as in above proposition yields a bijection $MV_{\mu_\bullet} \bij \BB_{\mu_\bullet}(\lambda)$ (cf.~\cite[Rmk.~3.3.16]{XiaoZhu:CyclesShVar}). The way the bijection between $MV_\mu(\lambda)$ and $\BB_\mu(\lambda)$ is constructed, the embedding $\BB_{\Sigma\,\mu_i} \mono \BB_{\mu_\bullet}$ corresponds to the map $MV_{\Sigma\, \mu_i}(\lambda) \mono MV_{\mu_\bullet}(\lambda), C \mapsto \cl{m_{\mu_\bullet}\iv(C)}$ (cf.~\cite[Thm.~3.3.18]{XiaoZhu:CyclesShVar}).


\section{Basic properties of affine Deligne-Lusztig varieties} \label{sect-ADLV}

\subsection{} \label{ssect-invariants}

 For any $b \in G(\L)$, we denote by 
 \[
  [b] = [b]_G \coloneqq \{gb\sigma(g)^{-1} \mid g \in G(\L)\}
 \]
 its $\sigma$-conjugacy class. By \cite{Kottwitz:GIsoc2}, every $\sigma$-conjugacy class $[b]$ is uniquely determined by two invariants: Its Newton point $\nu([b]) \in X_\ast(S)_{\QQ,{\rm dom}}$ and its Kottwitz point $\kappa([b]) \in \pi_1(G)_\Gamma$. We define $\lambda([b]) \in X_\ast(S)$ such that for every relative fundamental weight $\omega$ the fractional part of $\langle \omega,\lambda([b]) - \nu([b]) \rangle$ is contained in $(-1,0]$. Moreover, we denote by
 \[
  J_b(F) \coloneqq \{g \in G(\L) \mid b\sigma(g) = gb \}.
 \] 
 the twisted centraliser of $b$ inside $G$. It is equal to the $F$-valued points of the inner form of the centraliser $M_b$ of $\nu(b)$, obtained by twisting the Frobenius action with $b$. 
 
 \begin{defn}
  Let $\mu \in X_\ast(T)_{I,{\rm dom}}$ and $b\in G(\L)$. We denote by
  \[
   X_\mu(b) \coloneqq \{gK \in G(\L)/K \mid g^{-1}b\sigma(g) \in K\unif^\mu K\}
  \]
  the corresponding affine Deligne Lusztig variety. Moreover, we denote
  \[
   X_\mu^\lambda (b) \coloneqq X_\mu(b) \cap I\unif^\lambda K/K
  \]
 \end{defn}
  By \cite[Thm.~A]{He:KottwitzRapoportConj}, the variety $X_\mu(b)$ is nonempty if and only if the Mazur inequality
 \begin{align} \label{eq-Mazur}
   \nu([b]) &\leq \restr{\mu}{\Shat} \\
   \kappa([b]) &\equiv \mu \textnormal{ in } \pi_1(G)_\Gamma \nonumber
 \end{align}
 is satisfied.  Note that $J_b(F)$ acts on $X_\mu(b)$ by left multiplication and thus also on its set of irreducible components $\Irr(X_\mu(b))$. It is easy to see that if we replace $b$ by a $\sigma$-conjugate $b'$, we get compatible isomorphisms $J_b \cong J_{b'}$ and $X_\mu(b) \cong X_\mu(b')$.
 
 \begin{lem} \label{lemma-reduction-to-identity-component}
  Let $b,\mu$ as above such that $X_\mu(b) \not= \emptyset$. Then the set $\{ \omega \in \pi_1(G)_I \mid X_\mu(b)^\omega \not= \emptyset \}$ is a $\pi_1(G)_I^\sigma$-coset. In particular $J_b(F)$ acts transitively on this set and we have
  \[
   J_b(F)^0\backslash \Irr (X_\mu(b))^\omega \cong J_b(F)\backslash \Irr (X_\mu(b))
  \]
  for any $\omega \in \pi_1(G)_\sigma$. 
 \end{lem}
  \begin{proof}
   Applying the Kottwitz homomorphism $\kappa$ to $g\iv b\sigma(b) \in K\unif^\mu K$, we obtain
   \[
    -\kappa(g)+\kappa(b) + \sigma(\kappa(g)) = \kappa(\unif^\mu) \Lra (\sigma-1)\kappa(g) =  \kappa(\unif^\mu)-\kappa(b),
   \]
   proving that $\{ \omega \in \pi_1(G)_I \mid X_\mu(b)^\omega \not= \emptyset \}$ is contained in a simple $\pi_1(G)_I^\sigma$-coset. Since this set is stable under $J_b(F)$-action, the claims follow once we have shown that $\kappa$ maps  $J_b(F)$ onto $\pi_1(G)_I^\sigma$. Now $J_{b,L} \cong M_{b,L}$ induces a $\sigma$-equivariant isomorphism $\pi_1(J_b)_I \cong \pi_1(M_b)$. Since tha canonical morphism $\pi_1(M_b)_I^\sigma \to \pi_1(G)_I^\sigma$ is surjective, it therefore surfices to show that $\kappa_{J_b}$ maps $J_b(F)$ onto $\kappa(J_b)^\sigma_I$. This follows from \cite[Prop.~1.0.2]{HainesRostami:RamifiedSatake}.
  \end{proof}
 
 \begin{lem}  \label{lemma-reduction-to-adjoint}
  Let $b,\mu$ as above and let $\omega \in \pi_1(G)_I$ such that $X_\mu(b)^\omega \not= \emptyset$. Denote by $b_{\rm ad},\mu_{\rm ad},\omega_{\rm ad}$ the respective images of $b,\mu,\omega$ in the adjoint quotient $G_{\rm ad}$. Then the canonical map $\Flag_G \epi \Flag_{G_{\rm ad}}$ induces a universal homeomorphism
  \[
   X_\mu(b)^\omega \to X_{\mu_{\rm ad}}(b_{\rm ad})^{\omega_{\rm ad}},
  \]
  inducing a surjection
  \[
   J_b(F)\backslash \Irr (X_\mu(b)) \to J_{b_{\rm ad}}(F) \backslash \Irr (X_{\mu_{\rm ad}}(b_{\rm ad}))
  \]
 \end{lem}
 \begin{proof}
  The proof of the first assertion is standard (see e.g.\ \cite[Prop.~3.8]{Nie:ADLVIrredComp}). The second assertion  follows from the fact that the $J_b(F)$-action on $X_\mu(b)$ factors through $J_{b_{\rm ad}}(F)$.
 \end{proof}
 
 \begin{defn}
 We call $b$ basic if $\nu([b])$ is central. We call $b$ superbasic if for every rational Levi subgroup $M \subset G$ we have $[b] \cap  M(L) = \emptyset$.
\end{defn}

\subsection{} \label{ssect-superbasic} Since the intersection  $M_b \cap [b]$ is always non-empty, every superbasic $\sigma$-conjugacy class is also basic. More precisely, $b$ is superbasic if and only if $b$ is basic and $J_b$ is anisotropic modulo center (cf. \cite[Lemma~5.9.1]{GHKR:DimADL}). This implies that the simple factors of $J_b^{\rm ad}$ are inner forms of $\Res_{E/F}\PGL_n$ for some finite field extension $E/F$ by \cite[Thm~6.3]{Prasad:UnramDescent}. 
 
 \subsection{} An explicit description of the basic elements can be given as follows. Let $\Omega = \Omega_G \subset \Wtilde$ denote the stabiliser of $\afr$. Then the embedding $N(L) \mono G(L)$ induces a bijection between $\sigma$-conjugacy classes in $\Omega$ and basic $\sigma$-conjugacy classes in $G(L)$. More generally, for any $b\in G(L)$ we have that  $[b] \cap M_b(L)$ is a basic $\sigma$-conjugacy class and thus can by represented by an element of $\Omega_{M_b}$. We call such an element a \emph{standard representative} of $[b]$. Since the isomorphism class of $X_\mu(b)$ only depends on $\mu$ and $[b]$, we will henceforth assume that $b$ is a standard representative.
 
\section{Geometry of ADLVs for superbasic $b$} \label{sect-superbasic}

In this section we prove the main theorem in the case that $b$ is superbasic. For this we mainly follow Nie's strategy in the unramified case (\cite[\S~2]{Nie:ADLVIrredComp}). Moreover, the proofs of Lemma~\ref{lem-image-of-l_b} and Theorem~\ref{thm-superbasic} below are analogous to their counterparts \cite[Prop.~1.21, Lemma~2.5]{Nie:ADLVIrredComp} and \cite[Prop.~2.12, Thm.~2.1]{Nie:ADLVIrredComp}. Since Nie's strategy and notation is slightly different from ours, we give (or at least sketch) their proofs for the reader's convenience.

\subsection{}
 In order to reduce to the case that $\mu$ is minuscule later on, we consider a slightly more general definition of an affine Deligne-Lusztig variety. Let $d$ be a positive integer, $\mu_\bullet \in X_\ast(T)_{I,{\rm dom}}^d$ and $b \in G(L)$. We define $\sigma_\bullet\colon G^d(\L) \to G^d(\L), (g_1,g_2,\dotsc,g_d) \mapsto (g_2,\dotsc,g_{d},\sigma(g_1))$ and $b_\bullet \coloneqq (1,\dotsc,1,b) \in G^d(\L)$. We consider
 \begin{align*}
  X_{\mu_\bullet}(b) &\coloneqq \{ g\in G^d(\L) \mid g^{-1} b_\bullet \sigma_\bullet(g) \in K^d \unif^{\mu_\bullet} K \}
  \shortintertext{and for $\lambda_\bullet \in X_\ast(T)_I^d$}
  X_{\mu_\bullet}^{\lambda_\bullet}(b) &= X_{\mu_\bullet}(b) \cap I^d \unif^{\lambda_\bullet} K^d.
 \end{align*}
 
 \begin{lem} \label{lemma-reduction-to-EL-strata}
  Assume that $b$ is superbasic. Then the canonical map
  \begin{align*}
   (J_b(F)\cap I) \backslash \{ C \in \Irr X_{\mu_\bullet}^{\lambda_\bullet}(b) \mid \dim C = \dim X_{\mu_\bullet}(b), \kappa(\lambda_\bullet) = 0\} &\to J_b(F) \backslash  X_{\mu_\bullet}(b) \\ C &\mapsto \overline{C}
  \end{align*}
  is a bijection.
 \end{lem}
 \begin{proof}
  This follows from Lemma~\ref{lemma-reduction-to-identity-component}, as $J_b$ anisotropic implies that  $J_b(F)^0 = J_b(F) \cap I$. 
 \end{proof}
 
 In order to study the geometry of $X_{\mu_\bullet}^{\lambda_\bullet}(b)$, we consider the local model map on $I\unif^{\lambda_\bullet}$. For this we denote by $\lambda\p_\bullet \coloneqq  b_\bullet\sigma(\lambda_\bullet) - \lambda_\bullet$ and 
 \begin{align*}
  R_{\mu_\bullet,b}(\lambda_\bullet) \coloneqq & \{\alpha \in \Sigma_{G^d} \mid \alpha_\bullet \gg_{\lambda_\bullet} 0, \langle \alpha, \lambda\p_\bullet\rangle = -1 \}
 \end{align*}

 \begin{lem} \label{lem-image-of-l_b}
  Let $b = t^{\lambda_b} \cdot w_b \in \Omega_G$.
   \begin{subenv}
  \item For any $\lambda_\bullet \in X_\ast(T)_I^d$, the morphism $\ell_{b_\bullet}$ induces a pro\'etale cover
  \begin{equation} \label{eq-lb-on-Iwahori-translate}
   I^d\unif^{\lambda_\bullet} \to I^d_{\lambda_\bullet} \cdot \unif^{\lambda\p_\bullet} \cdot w_{b_\bullet}. 
  \end{equation}
  with Galois group $J_b(F) \cap I$ reflecting and preserving admissible subsets and their dimension.
  \item If $\mu_\bullet$ is minuscule then the intersection of $K^d\unif^{\mu_\bullet} K^d$ with the right hand side of (\ref{eq-lb-on-Iwahori-translate}) is non-empty if and only if $\lambda_\bullet\p \in W^d.\mu_\bullet$. In this case it is irreducible and smooth of dimension $R_{\mu_\bullet}(\lambda_\bullet) + \dim I^{\lambda_\bullet}$.
  \item If $\mu_\bullet$ is minuscule, then the image of $\ell_b(I^d\unif^{\lambda_\bullet}) \cap K^d\unif^{\mu_\bullet}K^d$ in $\Flag_{\mu_\bullet}$ equals the image of $U_{\lambda_\bullet}^d(0) \cdot \unif^{\lambda_\bullet\p}$
  \end{subenv}
 \end{lem}
 \begin{proof}(cf.~{\cite[Prop.~1.21, Lemma~2.5]{Nie:ADLVIrredComp}} for $G$ unramified)
  Since $b \in \Omega_G$, we have that $b_\bullet \eqqcolon \unif^{\lambda_{b,\bullet}}\cdot w_{b_\bullet} \in \Omega_{G^d}$. Now
  \[
   \ell_{b_\bullet}(i_\bullet\cdot \unif^{\lambda_\bullet}) = \unif^{-\lambda_\bullet} i\iv_\bullet b_\bullet \sigma_\bullet(i_\bullet) b\iv_\bullet b_\bullet \unif^{\sigma_\bullet(\lambda_\bullet)} = \unif^{-\lambda_\bullet} L_{b_\bullet}(i_\bullet) \unif^{b_\bullet\sigma_\bullet(\lambda_\bullet)} w_{b_\bullet},
  \]
  where $L_{b_\bullet}(i_\bullet) \coloneqq i_\bullet^{-1} \cdot b_\bullet\sigma_\bullet(i_\bullet)b_\bullet^{-1}$ denotes the Lang map. By Lang's lemma $L$ is a profinite \'etale covering of $I^d$  with Galois group $(J_b(F) \cap I)$. Hence $\restr{\ell_b}{I^d \unif^{\lambda_\bullet}}$ is a pro\'etale $J_b(F)\cap I$-torsor over its image $ I^d_{\lambda_\bullet} \cdot \unif^{\lambda\p_\bullet} \cdot w_{b_\bullet}$. By Lang's lemma $L_{b_\bullet}$ induces finite \'etale covers $I^d/I^d_n \to I^n/I^d_n$, thus it reflects and preserves admissible subsets of $I^d$ as well as their dimension. Now (1) follows from Lemma~\ref{lemma-admissible}.
  
  Now $I^d_{\lambda_\bullet} \cdot \unif^{\lambda\p_\bullet} \cdot w_{b_\bullet} \cap K^d \unif^{\mu_\bullet} K^d $ decomposes to
  \begin{equation} \label{eq-image-decomposition}
T^d(\L)_1 \cdot \underbrace{(I^d_{\lambda_\bullet} \cap \Ubar^d_{\lambda_\bullet}(\L))}_{\in K_1} \cdot \left( (I^d_{\lambda_\bullet} \cap U^d_{\lambda_\bullet}(\L)) \unif^{\lambda\p_\bullet}) \cap K^d \unif^{\mu_\bullet} K^d \right) \cdot w_{b_\bullet}.
  \end{equation}
  Now assume that $\mu$ is minuscule. By \cite[Lemma~10.2.1]{HainesRostami:RamifiedSatake} the intersection $U_{\lambda_\bullet}^d \unif^{\lambda\p_\bullet} \cap K^d \unif^{\mu_\bullet} K^d $ is non-empty if and only if $\lambda\p_\bullet \in W^d.\mu$. Then the same calculation as in the split case (\cite[Lemma~5.2]{NgoPolo:CasselmanShalika}) shows that the intersection equals  $\prod_{\alpha >_{\lambda_\bullet} 0} U'_{\alpha,0} \unif^{\lambda\p_\bullet} \prod_{\alpha >_{\lambda_\bullet} 0} U'_{\alpha,0} = \prod_{\alpha >_{\lambda_\bullet} 0} U'_{\alpha,\min\{\langle \alpha, \lambda\p_\bullet \rangle,0\}}$. Assuming that $\lambda\p_\bullet \in W^d.\mu_\bullet$, the set (\ref{eq-image-decomposition}) equals
  \[
   T^d(L)_1 \cdot (I^d_{\lambda_\bullet} \cap K^d) \cdot (I^d_{\lambda_\bullet} \cap  \prod_{\alpha >_{\lambda_\bullet} 0} U'_{\alpha,\min\{\langle \alpha, \lambda\p_\bullet \rangle,0\}}) \cdot w_b.
  \]
  Since $\mu_\bullet$ is minuscule, so is $\lambda\p_\bullet$. Hence
  \[
   (I_{\lambda_\bullet} \cap K^d) \backslash I^d_{\lambda_\bullet} \cdot \unif^{\lambda\p_\bullet} \cdot w_{b_\bullet} \cap K^d \unif^{\mu_\bullet} K^d = \prod_{\alpha \in R_{\mu_\bullet,b}(\lambda_\bullet)} U'_{\alpha,0} \backslash U'_{\alpha,-1} \cong \AA^{R_{\mu_\bullet,b}(\lambda_\bullet)}.
  \]
  Since $\dim_K  (I^d_{\lambda_\bullet} \cap K^d) = \dim_K I^{d, \lambda_\bullet}$ by Lemma~\ref{lemma-admissible}, the second claim follows.
  
 To prove the third claim, we first note that
 \begin{equation} \label{eq 4.3.3}
  \unif^{-\lambda_\bullet\p}\prod_{\alpha \in R_{\mu_\bullet,b}(\lambda_\bullet) } U_{\alpha,-1} \unif^{\lambda_\bullet\p} = \prod_{\alpha >_{b_\bullet\sigma_\bullet} 0 \atop \langle \alpha, \lambda_\bullet\p \rangle = -1} U_{\alpha,0} \subset I_{b_\bullet \sigma_\bullet}^d \cap U_{\lambda_\bullet}^d(0).
 \end{equation}  
  and that since $\lambda_\bullet\p$ is minuscule
  \begin{equation} \label{eq 4.3.4}
   I_{b_\bullet\sigma_\bullet(\lambda_\bullet)}^d \cap \Ubar^d_{\lambda_\bullet} = \unif^{-\lambda_\bullet\p} \underbrace{I_{\lambda_\bullet}^d \cap \Ubar^d_{\lambda_\bullet}}_{\subset K_1} \unif^{\lambda_\bullet\p} \subset K.
  \end{equation}
   Also note that by definition $I^d_{b\bullet \sigma_\bullet(\lambda_\bullet)} = I_{\lambda_2} \times \dotsb I_{\lambda_d} \times I_{b\sigma(\lambda_1)}$.
  By above calculations, the image of $\ell_b(I^d\unif^{\lambda_\bullet}) \cap K^d\unif^{\mu_\bullet}K^d$ in $\Flag_{\mu_\bullet}$ equals
  \begin{align*}
   & (I_{\lambda_1}\cap K ) \prod_{\alpha \gg_{\lambda_1} 0 \atop \langle \alpha \lambda_1\p \rangle = -1 } U_{\alpha,-1} \unif^{\lambda\p_1} \times_K \dotsb  \times_K (I_{\lambda_{d}} \cap K) \prod_{\alpha \gg_{\lambda_{d}} 0 \atop \langle \alpha \lambda_d\p \rangle = -1 } U_{\alpha,-1} \unif^{\lambda\p_d} K \\
   \stackrel{(\ref{eq 4.3.3})}{=} & I_{\lambda_1}\cap K \cdot \unif^{\lambda\p_1} \times_K \dotsb  \times_K (I_{\lambda_{d}} \cap K) \unif^{\lambda\p_d} K \\
   \stackrel{(\ref{eq 4.3.4})}{=} & (I_{\lambda_1} \cap K \cap U_{\lambda_1}) \cdot \unif^{\lambda\p_1} \times_K \dotsb  \times_K (I_{\lambda_{d}} \cap K \cap U_{\lambda_d} ) \unif^{\lambda\p_d} K \\
   = \enspace\, & (K \cap U_{\lambda_1}) \cdot \unif^{\lambda\p_1} \times_K \dotsb \times_K (K \cap U_{\lambda_d}) \unif^{\lambda\p_d} K.
  \end{align*}

\end{proof}

  \begin{cor} \label{cor-reduction-to-adjoint-sb}
  Assume that $b$ is superbasic. Then the canonical map
  \[
   J_b(F)\backslash \Irr (X_{\mu_\bullet}(b)) \to J_{b_{\rm ad}}(F) \backslash \Irr (X_{\mu_{\bullet,\rm ad}}(b_{\rm ad}))
  \]
  is a bijection.
 \end{cor}
 \begin{proof}
  By Lemma~\ref{lemma-reduction-to-EL-strata} and part (1) of the previous lemma we may equivalently consider the map
  \[
   \Irr(I^d_{\lambda_\bullet} \cdot \unif^{\lambda\p_\bullet} \cdot w_{b_\bullet} \cap K^d \unif^{\mu_\bullet} K^d ) \to \Irr(I^d_{\mathrm{ad},\lambda_{\bullet,\rm ad}} \cdot \unif^{\lambda_{\bullet,\rm ad}\p} \cdot w_{b_\bullet, {\rm ad}} \cap K_{\rm ad}^d \unif^{\mu_{\bullet,{\rm ad}}} K_{{\rm ad}}^d )
  \]
  for $\lambda_\bullet$ with $\kappa(\lambda_\bullet) = 0$. But the above map is a bijection by \cite[Prop.~3.1]{HamacherViehmann:ADLVIrredComp}. 
 \end{proof}

 By the  previous lemma, the $J_b(F) \cap I$-orbits of topdimensional irreducible components of $X_{\mu_\bullet}(b)$ for minuscule $\mu_\bullet$ are in canonical bijection with $\lambda_\bullet \in X_\ast(T^d)$ such that $\#R_{\mu_\bullet,b}(\lambda_\bullet)$ is maximal. As the next step we conclude a simple formula for $\# R_{\mu_\bullet,b}(\lambda_\bullet)$,
 for which we denote $\tilde\lambda_\bullet \coloneqq w_{\lambda_\bullet}\iv(\lambda_\bullet\p)$.
 
 \begin{lem} \label{lemma-R}
  Let $\mu_\bullet \in X_\ast(T^d)$ be minuscule and let $\lambda_\bullet \in W^d.\mu$. Then
  \[
   \# R_{\mu_\bullet, b}(\lambda_\bullet) = \langle \rho_{G^d}, \tilde\lambda_\bullet - \mu_{\bullet, {\rm adom}} \rangle.
  \]
 \end{lem}
 \begin{proof}
  We have
  \begin{align*}
   &\# R_{\mu_\bullet, b}(\lambda_\bullet) + \dim I^{d,\lambda_\bullet} \\ &\stackrel{L.~\ref{lem-image-of-l_b}}{=}  \dim I^d_{\lambda_\bullet} \unif^{\lambda\p}  
   \stackrel{L.~\ref{lemma-adm-inv}}{=} \dim \unif^{-\lambda\p}I^d_{\lambda_\bullet} 
   = \dim I^d_{b_\bullet\sigma_\bullet(\lambda_\bullet)} \unif^{-\lambda\p} \\
   &\,\,\,\,=\,\#\{\alpha \in \Sigma \mid \alpha \gg_{b_\bullet\sigma_\bullet(\lambda_\bullet)} 0,  \langle \alpha, -\lambda_\bullet\p\rangle = -1 \} + \dim I^{d,b_\bullet\sigma_\bullet(\lambda_\bullet)},
  \end{align*}
  where the last equality follows by the same calculations as in the previous lemma. Since $I^{d,\lambda_\bullet} = I^{d,b_\bullet\sigma_\bullet(\lambda_\bullet)}$, we obtain
  \begin{equation} \label{eq-R}
   \# R_{\mu_\bullet, b}(\lambda_\bullet) = \#\{\alpha \in \Sigma_{G^d} \mid \alpha \gg_{b_\bullet\sigma_\bullet(\lambda_\bullet)} 0,  \langle \alpha, \lambda_\bullet\p\rangle = 1 \}.
  \end{equation}
  Note that if $\langle \alpha, \lambda_\bullet\p\rangle = 1$, we have equivalences
  \[
   \alpha \gg_{b_\bullet\sigma_\bullet(\lambda_\bullet)} 0 \Lra \alpha >_{\lambda_\bullet} 0 \Lra w_{\lambda_\bullet}\iv (\alpha) > 0.
  \]
  Thus we deduce from (\ref{eq-R}) that
  \[
   \# R_{\mu_\bullet, b}(\lambda_\bullet) = \#\{\alpha \in \Sigma_{G^d} \mid \alpha > 0,  \langle \alpha, \tilde\lambda_\bullet\rangle = 1 \}.
  \]
  Now denote $N_i \coloneqq \#\{\alpha \in \Sigma_{G^d} \mid \alpha > 0,  \langle \alpha, \tilde\lambda_\bullet\rangle = i \}$. Since $\tilde\lambda_\bullet \in W^d.\mu_\bullet$, we have
  \begin{align*}
   \langle \rho_{G^d}, \mu_{\bullet, {\rm adom}} \rangle &= -\frac{1}{2}(N_{1}+N_{-1}) \\
   \langle \rho_{G^d}, \tilde\lambda_\bullet\rangle &= \frac{1}{2} (N_1- N_{-1})
   \shortintertext{and thus}
   \langle \rho_{G^d},\tilde\lambda_\bullet -  \mu_{\bullet, {\rm adom}} \rangle &= N_{1} = \# R_{\mu_\bullet, b}(\lambda_\bullet)
  \end{align*}
 \end{proof}

 As a consequence we obtain the main theorems for superbasic $b$ and minuscule $\mu$.
 
 \begin{prop}
 Let $b = t^{\lambda_b} \cdot w_b \in \Omega_G$ and $\mu_\bullet$ be minuscule. 
 \begin{subenv}
  \item $X_{\mu_\bullet}^{\lambda_\bullet}(b)$ is non-empty if and only if $\lambda_\bullet\p \in W^d.\mu_\bullet$. In this case it is  smooth of dimension $R_{\mu_\bullet,b}(\lambda_\bullet)$ and $J_b(F) \cap I$ acts transitively on its components.
  \item Assume in addition that $b$ is superbasic. Then the map
  \begin{align*}
   \{\lambda_\bullet \in X_\ast(T)_I^d \mid X_{\mu_\bullet}^{\lambda_\bullet}(b) \not= \emptyset\} &\to \{\tilde\mu_\bullet \in W^d.\mu_\bullet \mid  \restr{\tilde\mu_\bullet}{\Shat} \geq \nu([b]) \} \\
   \lambda_\bullet &\mapsto \tilde\lambda_\bullet  
  \end{align*}
  is surjective and its fibres are precisely the $\Omega^\sigma$-orbits. In particular, the topdimensional pieces correspond to those $\tilde\lambda_\bullet$ restricting to $\lambda([b])$ and
  \[
   \dim X_{\mu_\bullet}(b) = \langle \rho_{G^d},\mu_\bullet - \nu([b]) \rangle - \frac{1}{2} \defect_G([b])
  \]
  and have a canonical bijection
  \[
   \# J_b(F) \backslash \Irr X_{\mu_\bullet}(b) \bij \BB_{\mu_\bullet}(\lambda([b])),
  \]
  satisfying the description of Proposition~\ref{prop-ic}.
 \end{subenv} 
 \end{prop} 
 \begin{proof}
  (1) By the previous lemma we have
  \[
   \dim X_{\mu_\bullet}^{\lambda_\bullet}(b) = \dim I_{\lambda_\bullet} \cdot \unif^{\lambda\p_\bullet} \cdot w_{b_\bullet} \cap K^d \unif^{\mu_\bullet} K^d  - \dim I^{\lambda_\bullet} = \# R_{\mu_\bullet,b}(\lambda_\bullet)
  \]
  and similarly
  \[
   J_b(F) \backslash \Irr X_{\mu_\bullet}^{\lambda_\bullet}(b) \cong \Irr I_{\lambda_\bullet} \cdot \unif^{\lambda\p_\bullet} \cdot w_{b_\bullet} \cap K^d \unif^{\mu_\bullet} K^d,
  \] 
  which is a singleton.
  
 (2) By Corollary~\ref{cor-reduction-to-adjoint-sb}, the statement does not change when we replace $G$ by a group with isomorphic adjoint group. Hence we may assume that $G = \Res_{E/F} \GL_n$ by (\ref{ssect-superbasic}). We denote by $S \subset G$ the diagonal torus and by $B$ the Borel subgroup of lower triangular matrices. We denote by $f$ the inertia index of $E/F$ and identify $G(\L) = \GL_n(\Ebreve)^f$. Then we may choose $b$ as follows: There exist integers $m_1,\dotsc,m_f$ such that for every $1 \leq \tau \leq f$ we have
 \[
  b_\tau = \diag(\underbrace{\unif^{\lceil m_\tau/n \rceil},\dotsc,\unif^{\lceil m_\tau/n \rceil}}_{(m_\tau\ {\rm mod}\ n) \ times}, \unif^{\lfloor m_\tau/n \rfloor},\dotsc,\unif^{\lfloor m_\tau/n \rfloor}) \cdot P_{m_\tau}
 \]
 where $P_{m_\tau}$ is the permutation matrix corresponding to the permutation $i \mapsto i+m_\tau \ {\rm mod} \ n$.
 
 Now the claim follows from the calculations in \cite[\S~4]{HamacherViehmann:ADLVIrredComp}. We briefly explain how to translate our description to  the language of EL-charts used in \cite{HamacherViehmann:ADLVIrredComp}.  We denote by $f$ the inertia index of $E/F$ and identify $X_\ast(T^d)_I = ((\ZZ^n)^f)^d\cong (\ZZ^n)^{d\cdot f}$ such that $\sigma_\bullet(v_1,v_2,\dotsc,v_{df}) = (v_{d\cdot f},v_1,\dotsc,v_{d\cdot f-1})$ and $b_\bullet$ only acts non-trivial on the $d,2d,\dotsc,d\cdot f$-th component. This also identifies $W^d = S_n^{d\cdot f}$. For any $\lambda_\bullet \in X_\ast(T)_I^d = (\ZZ^n)^{d\cdot f}$, we define
 \[
  A(\lambda_\bullet) \coloneqq \bigsqcup_{\tau=1}^{d\cdot f-1} \{ x\cdot n + i \mid 1 \leq i \leq n, x \geq \lambda_{\tau,i} \}.
 \] 
 Moreover, we denote $\ZZ^{(d\cdot f)} \coloneqq \bigsqcup_{\tau=1}^{d\cdot f} \ZZ$ and by for any $m \in \ZZ$ we denote by $k_{\tau} \in \ZZ^{(d\cdot f)}$ the corresponding element in the $\tau$-th component. We consider
 \begin{align*}
  \frm\colon \ZZ^{(d\cdot f)} &\to \ZZ^{(d\cdot f)} \\
  k_{(\tau-1)} &\mapsto \begin{cases}
   (k+m_\tau)_{(\tau)} &\textnormal{if } d \mid \tau \\
   k_{(\tau)} & \textnormal {else.}
  \end{cases}
 \end{align*}
 Then $A(b_\bullet\sigma_\bullet(\lambda_\bullet)) = \fsf(A(\lambda_\bullet))$.
 
 Now $A(\lambda_\bullet)$ is a weak EL-chart in the sense of \cite{ChenViehmann:ADLV}, i.e.\ it is bounded below, $A(\lambda_\bullet)+n \subset A(\lambda_\bullet)$ and $i_{(\tau)} \in A(\lambda_\bullet)$ for all $\tau$ if $i$ is big enough. This is slightly more general than the notion of an EL-chart in \cite{HamacherViehmann:ADLVIrredComp}, but does not change the combinatorics. Comparing definitions, we see that the type of $A(\lambda_\bullet)$ equals the image of $\lambda_\bullet\p$ under a certain permutation $w \in W^d$. More precisely we have
 \[
  \type(A) = (\varsigma_{m,n}, \dotsc, \varsigma_{m,n}) \circ (c_{\lambda_1},\dotsc,c_{\lambda_1})(\lambda_\bullet\p),
 \]
 where $c_{\lambda_1} \in S_n$ is the cyclic permutation mapping the largest entry of $\lambda_1$ to the first position and $(\varsigma_{m,n}\colon i \mapsto i\cdot m\ \mathrm{mod}\ n) \in S_n$.  In particular the Hodge point of $A(\lambda)$ equals $\type(A)_{\rm dom} = (\lambda_\bullet\p)_{\rm dom}$. 
 
 Similarly one sees that the cotype of $A(\lambda_\bullet)$ equals $w_0 \tilde\lambda_\bullet$\footnote{ Since we choose the upper triangular Borel in \cite{HamacherViehmann:ADLVIrredComp}, the $w_0$ occurs here}.
We recall that \cite[Thm.~4.16]{HamacherViehmann:ADLVIrredComp} the cotype has defines a bijection
 \[
  \{\textnormal{EL-charts with Hodge-point } \mu_\bullet\}/\sim \bij \{\tilde\mu_\bullet \in W^d.\mu_\bullet \mid  \restr{\tilde\mu_\bullet}{\Shat} \leq \nu(b) \},
 \] 
 where $A_1 \sim A_2$ iff there exists $k \in \ZZ$ such that $A_2 = A_1+k$. One easily checks that this coincides with the action of $\Omega^\sigma$, proving the first assertion of (2). 
 
 Since $J_b$ is anisotropic mod center, its Iwahori decompostion reduces to $J_b(F) = \Omega^\sigma \cdot (J_b(F) \cap I)$. Thus we deduce from (1) that
 \[
  \dim X_{\mu_\bullet}(b) = \max \# R_{\mu_\bullet,b}(\lambda_\bullet) 
 \]
 and that the $J_b(F)$-orbits of topdimensional irreducible components correspond to $\tilde\lambda_\bullet$ with $\#R_{\mu_\bullet,b}(\lambda_\bullet)$ maximal. By Lemma~\ref{lemma-R} this is equivalent to $\restr{\tilde\lambda_\bullet}{\Shat^d}$ being maximal, i.e.\ equal to $\lambda([b])$. In this case we obtain
 \[
  \dim X_{\mu_\bullet}(b) = \langle \rho_{G^d}, \tilde\lambda([b]) - \mu_{\bullet, {\rm adom}} \rangle = d(b,\mu),
 \]
 where the last equality follows from \cite[Lemma~4.6]{Hamacher:DimADLV}, which is a purely root theoretic computation and thus also holds in the ramified case.  
 \end{proof}

We can now prove the main theorem in the case that $b$ is superbasic.

 \begin{thm} \label{thm-superbasic}
  Assume that $b \in G(\L)$ is superbasic. 
  \begin{subenv}
   \item We have
  \[
   \dim X_\mu(b) = d(b,\mu)
  \]
  and
  \[
   \# J_b(F) \backslash \Irr X_\mu(b) = \dim V_\mu(\lambda(b)).
  \]
  \item Explicitely, let $C \in \Irr X_\mu(b)$ and $\lambda_C$ such that $C \subset \cl{X_\mu^{\lambda_C}(b)}$. Then
  \[
   Z_C \coloneqq w_{\lambda_C}\iv \cl{\ell_b( \tilde{C} \cap I\unif^{\lambda_C}) \cdot K /K} \in MV_\mu(\tilde\lambda_C)
  \]
  and $C \mapsto Z_C$ defines a bijection 
  \[
   J_b(F) \backslash \Irr X_\mu(b) \bij \bigcup_{\restr{\tilde\lambda}{T^\vee} = \lambda([b])} MV_\mu(\tilde\lambda)
  \]  
  
  \end{subenv}
 \end{thm} 
 \begin{proof}
 The proof follows by the same argument as in the unramified case (\cite[Prop.~2.12, Thm.~2.1]{Nie:ADLVIrredComp}). Since we are not assuming the dimension formula a priori as in loc.~cit., we give the proof for the reader's convenience. 
  
  By Lemma~\ref{lemma-reduction-to-adjoint} and the previous corollary it we may replace $G$ by $G^{\rm ad}$ and vice versa. Hence we may assume that $G  = \Res_{E/F} GL_n$ by \eqref{ssect-superbasic}. In particular, we may write $\mu = \mu_1 + \dotsc + \mu_d$ as a sum of minuscule cocharacters. Note that the projection $G^d \to G$ to the first factor induces a $J_b(F)$-equivariant surjective morphism $m\colon X_{\mu_\bullet}(b) \to X_{\leq\mu}(b)$. By the same argument as in the unramified case \cite[Thm.~1.22,Cor~1.23]{Nie:ADLVIrredComp} we see that for $\mu' \leq \mu$ and $C \in \Irr X_{\mu'}(b)$, we have that $m\iv (C)$ is equidimensional of dimension $\dim C + \langle \rho, \mu - \mu' \rangle$ and has $m_{\mu_\bullet}^\mu$ irreducible components. 

  We proceed by induction on $\mu$. Assume the theorem holds true for every $\mu' < \mu$. By the dimension formula the closure of the preimage of each $C \in \Irr X_{\mu'}(b)$ contains $m_{\mu_\bullet}^{\mu'}$ topdimensional irreducible components of $X_{\mu_\bullet}(b)$. Hence the number of $J_b(F)$-orbits of irreducible components of $X_{\mu_\bullet}(b)$ mapping generically to $X_\mu(b)$ equals
 \[
  \dim V_{\mu_\bullet}(\lambda(b)) - \sum_{\mu' < \mu} \dim V_{\mu'}(\lambda(b)) = \dim V_\mu(\lambda(b)) 
 \]
 In particular, we obtain $\dim X_\mu(b) = \dim p^{-1} X_\mu(b) = d(b,\mu)$. Since $m_{\mu_\bullet}^\mu = 1$ we obtain
 \[
  \# J_b(F) \backslash \Irr(X_\mu(b)) = \# J_b(F) \backslash \Irr(p^{-1}(X_\mu(b))) =  V_\mu(\lambda(b)).
 \]
  Now consider the Cartesian diagram
  \begin{center}
   \begin{tikzcd}
   \Xtilde_{\mu_\bullet}(b) \ar{r}{\ell_{b_\bullet}} \ar{d}{pr_1} & \left[ b_\bullet \right] \cap K^d\unif^{\mu_\bullet}K^d \ar{d}{mult.}\\ 
   \Xtilde_{\leq \mu}(b) \ar{r}{\ell_b}  & \left[ b \right] \cap \bigcup_{\mu' \leq \mu} K \unif^{\mu'} K.
   \end{tikzcd}
  \end{center}
  By the above reasoning, there exists a unique $C_\bullet \in \Irr X_{\mu_\bullet}(b)$ with $\cl{pr_1\iv(C)} = C_\bullet$. We fix $\lambda_\bullet \in X_\ast(T)^d$ such that $C_\bullet = \cl{X_{\mu_\bullet}^{\lambda_\bullet}(b)}$. By same calculations as in the unramified case (\cite[Prop.~2.4]{Nie:ADLVIrredComp}), we deduce from Lemma~\ref{lem-image-of-l_b}~(3) that the image of $\ell_{b_\bullet}(I^d\unif^{\lambda_\bullet}) \cap K^d\unif^{\mu_\bullet}K^d$ in $\Flag_{\mu_\bullet}$ is a dense open subset of the image of $w_{\lambda_\bullet} S_{\tilde\lambda_\bullet} \cap \Flag_{\mu_\bullet}$. Altogether, we get a commutative diagram
    \begin{center}
     \begin{tikzcd}
      pr_1\iv(\Ctilde) \cap I \unif^{\lambda_\bullet} \ar{d}{pr_1} \ar{r}{\ell_{b_\bullet}} & I_{\lambda_\bullet} \unif^{\lambda_\bullet\p} \cap \tilde{m}_{\mu_\bullet}\iv(K\unif^\mu K) \ar{d}{\tilde{m}_{\mu_\bullet}} \ar{r} & w_{\lambda_C} S_{\tilde\lambda_\bullet} \cap m_\mu\iv(\Flag_{\mu}) \ar{d}{m_{\mu_\bullet}} \\
      \Ctilde \cap I  \unif^{\lambda_C} \ar{r}{\ell_b} & I_\lambda \unif^{\lambda_\bullet\p} \cap K\unif^\mu K \ar{r} & w_{\lambda_C} S_{\tilde\lambda_C} \cap \Flag_\mu
     \end{tikzcd}
    \end{center}
    where $\tilde{m}_{\mu_\bullet}\colon K^d \unif^{\mu_\bullet} K^d \to G(L)$ denotes the multiplication map. By the reasoning above, the top right morphism is dominant. As $\restr{\tilde\lambda_C}{S^\vee} = \restr{\tilde\lambda_\bullet}{S^\vee} = \lambda([b])$, the map $J_b(F) \backslash \Irr X_\mu(b)\to \bigcup MV_\mu(\tilde\lambda), C \mapsto Z_C$ is well-defined. Since to any $Z \in\bigcup MV_\mu(\tilde\lambda) $ there exists a unique $\tilde\lambda_\bullet$ such that ${m_{\mu_\bullet}}\iv(Z) = S_{\tilde\lambda_\bullet} \cap m_{\mu_\bullet}\iv(\Flag_\mu)$ (cf.~\ref{ssect-mv-tensor}), this map is bijective. 
 \end{proof}

\section{Reduction to a Levi subgroup} \label{sect-reduction-step}

\subsection{} \label{ssect-reduction-to-levi} Let $P = MN \subset G$ be a (rational) standard Levi subgroup where $M$ its Levi factor and $N$ its unipotent radical. For any subgroup $H \subset G(L)$ we denote $H_P \coloneqq H \cap P$ and $H_M \coloneqq H \cap M$. Then we have a diagram of ind-schemes
\begin{center}
 \begin{tikzcd}
  & \Flag_P \arrow{ld}{p} \arrow{rd}{i} & \\
  \Flag_M& & \Flag_G.
 \end{tikzcd}
\end{center}
We denote for any $b \in M(L)$ by $X_\mu(b)$ the affine Deligne-Lusztig variety in $\Flag_G$ and
\begin{align*}
 X_\mu^{P\subset G}(b) &\coloneqq i^{-1} X_\mu(b) \\
 X_\mu^{M\subset G}(b) &\coloneqq p(X_\mu^{P\subset G}(b)) \\ &= \{ gK_M \in M(L)/K_M \mid g\iv b\sigma(b) \in K \unif^\mu K \} \\&= \bigsqcup_{\mu_M \in  X_\ast(T)_{M-dom} \atop \mu_{M,dom} = \mu} X_{\mu_M}^M(b).
\end{align*}
We denote by $X_\ast(T)_{I,M-dom} \subset X_\ast(T)$ the submonoid of $M$-dominant elements and denote by $\leq_M$ the Bruhat order on  $X_\ast(T)_{I,M-dom}$. We define $I_{\mu,b} \coloneqq \{\mu_M \in X_\ast(T)_{M-dom} \mid \mu_{M,dom} = \mu \textnormal{ and } X_{\mu_M}^M(b) \not= \emptyset \}$.

\subsection{} \label{sect-XP} As in the unramified case (\cite[p.~1628]{Hamacher:DimADLV}) one proves that $i$ identifies $\Flag_P$ with a disjoint union of locally closed subsets of $\Flag_G$ which cover $\Flag_G$. In particular, we obtain
\begin{align*}
 \dim X_\mu(b) &= \dim X_\mu^{P\subset G}(b) \\
 \Irr X_\mu(b) &= \Irr X_\mu^{P\subset G}(b)
\end{align*}
Thus it remains to relate the geometry of $X_\mu^{P\subset G}(b)$ with $X_\mu^{M\subset G}(b)$. For this we generalise the calculations of \cite[\S~5]{GHKR:DimADL} in the split case. Denote by $\pi\colon X_\mu^{P\subset G}(b) \to X_\mu^{M\subset G}(b)$ the restriction of $p$. For any $mK_M \in X_\mu^{M\subset G}(b)$ we have
 \begin{align*}
  \pi\iv(m) &= \{ mnK_P \in P(L)/K_P \mid n\iv m\iv b \sigma(m) \sigma(n) \in K \unif^\mu K \} \\
  &\cong \{ nK_N \in N(L)/K_N \mid L_{1,b_m}(n) \in K\unif^\mu Kb_m \cap N(L) \},
 \end{align*}
 where $b_m \coloneqq m\in b \sigma(m)$ and $L_{1,b_m}(n) \coloneqq n\iv \cdot b_m \sigma(n) b_m\iv$. Writing $b_M = k_1 \unif^{\mu_M} k_2$ with $k_1,k_2 \in K_M, \mu_M \in I_{\mu,b}$, conjugating with $\unif^{\mu_M}k_2$ gives us a bijection
 \[
  K\unif^\mu Kb_m \cap N(L) \cong K\unif^\mu K \cap N(L)\unif^{\mu_M}.
 \]
 We denote
 \[
  d(\mu,\mu_M) \coloneqq \dim \{ n \in N(L)/K_N \mid n\unif^{\mu_M} \in K\unif^\mu K \}.
 \]
 \begin{defn}
  For any $\mu \in X_\ast(T)_{I,dom}, \kappa \in \pi_1(M)_\Gamma$ we denote
  \begin{align*}
   S_M(\mu) &= \{\mu_M \in X_\ast(T)_{I,M-dom} \mid K\unif^\mu K\unif^{\mu_M} \cap N(L) \not= \emptyset \} \\
   \Sigma(\mu) &= \{\mu' \in X_\ast(T)_I \mid \mu'_{dom} \leq \mu \} \\
   \Sigma(\mu)_{M-dom} &= \Sigma(\mu) \cap X_\ast(T)_{I,M-dom} \\
   \Sigma(\mu)_{M-max} &= \{\mu' \in \Sigma(\mu)_{M-dom} \mid \mu' \textnormal{ maximal  in } \Sigma(\mu)_{M-dom} \textnormal{ w.r.t.\ } \leq_M \}.
  \end{align*}
 \end{defn}
 
  \begin{prop}[{\cite[Lemma~5.4.1,Prop.~5.4.2]{GHKR:DimADL}}] \label{prop-intersection}
   We have
   \[
    \Sigma(\mu)_{M-max} \subset S_M(\mu) \subset \Sigma(\mu)_{M-dom}.
   \]
   For any $\mu_M \in S_M(\mu)$ we have  $d(\mu,\mu_M) \leq \langle \rho, \mu+\mu_M \rangle - 2 \langle \rho_M, \mu_M \rangle$. Moreover, we have equality if and only if $V_{\mu_M}^M$ occurs in the restriction of $V_\mu$ to $M^\vee$. In this case the number of top-dimensional irreducible components equals the multiplicity of $V_{\mu_M}^M$ in $\restr{V_\mu}{M^\vee}$.
   \end{prop}
   \begin{proof}
    The proof of the first statement is identical to the proof in \cite{GHKR:DimADL}, as it is derived from the statement that $U(L)\unif^{\mu'} \cap K\unif^\mu K \not= \emptyset$ if and only if $\mu' \in \Sigma(\mu)$ (\cite[Lemma~10.2.1]{HainesRostami:RamifiedSatake}).
    
   The proof of the second statement is also identical to the proof in \cite{GHKR:DimADL}, as it follows via formal calculations from the Lusztig-Kato formula. The Lusztig-Kato formula for quasi-split groups is deduced in \cite[eq.~(7.9)]{Haines:Dualities} using the general framework of Knop \cite{Knop:KazhdanLusztigBasis}. 
   \end{proof}
  
 This proposition has the following geometric interpretation. We denote by $\Flag_{G,\mu} = K \unif^\mu K / K$ and $S_\lambda = U(\L)\unif^\lambda K / K$ resp.\ $S_\lambda^M = U_M(\L) \unif^\lambda K_M/K_M$.
 \begin{prop} \label{prop-intersection-geom}
  The image of the morphism $p_{\lambda,\mu} \coloneqq \restr{p}{S_\lambda \cap \Flag_{G,\mu}}$ equals  $\bigcup_{\mu_M \in S_M(\mu)} S_\lambda^M \cap \Flag_{M,\mu_M}$.
  Moreover, there exists a canonical bijection
  \[
   p_{\lambda,\mu}^{MV} \colon MV_\mu(\lambda) \bij \bigcup_{\mu_M \in S_M(\mu) \atop V_{\mu_M}^M \textnormal{occurs in} \restr{V_\mu}{M^\vee}} MV_\mu^M(\lambda) \times \Irr(N(\L)\unif^{\mu_M} \cap K\unif^\mu K)
  \]
  such that for $C \in MV_{\mu_M}^M(\lambda)$, we have $p_{\lambda,\mu}\iv( {C} \times \Irr(N(\L)\unif^{\mu} \cap K\unif^\mu K)) = \Irr \cl{p_{\lambda,\mu}\iv(C)}$
 \end{prop}
 
 The main ingredient of the proof is the following lemma. 
 
 \begin{lem} \label{lemma-fibration}
  Let $\mu,\mu_M \in X_\ast(T)$ and $p_\mu\colon P(L) \cap K\unif^\mu K \to M(L)$ the canonical projection. There exists a Zariski cover  $V \to K_M\unif^{\mu_M}K_M$, such that 
  \[
   p_{\mu}\iv (K_M\unif^{\mu_M}K_M)_V \cong (N(L)\unif^{\mu_M}K \cap K \unif^\mu K) \times V. 
  \]
  This isomorphism is canonical up to postcomposing with $(K_M \cap \unif^{\mu_M} K_M \unif^{-\mu_M})$. In particular we obtain a canonical bijection
  \[
   \Irr(p_\mu\iv(Z)) \bij \Irr(Z) \times \Irr (N(L)\unif^{\mu_M}K \cap K \unif^\mu K)
  \]
 \end{lem} 
 \begin{proof}
  By \cite[Rmk.~5.7]{HamacherViehmann:ADLVIrredComp} (see also \cite[Lemma~2.1]{HartlViehmann:Foliation}), there exists a Zariski cover $U \to  K_M\unif^{\mu_M}K_M$ such that the corresponding element $x \in LM(U)$ decomposes as $x = y_1 \unif^{\mu_M} y_2$ with $y_1,y_2 \in L^+\Kscr_M(U)$. For any $U$-scheme $S$ and $z \in LN(S)$, we have
  \[
   z\cdot y_1 \unif^{\mu_M} y_2 \in (K \unif^\mu K)(S) \Lra (y_1\iv z y_1) \cdot \unif^{\mu_M} \in (K \unif^\mu K)(S).
  \]
  Thus $g \mapsto y_1\iv g$ yields an isomorphism as claimed. Note that it only depends on the choice of $k_1$. For any other choice $k_1'$, we have $k_1'k_1\iv \in (K_M \cap \unif^{\mu_M} K_M \unif^{-\mu_M})(U)$. Thus the isomorphisms constructed in Claim~\ref{claim-cover} differ by postcomposing with $(K_M \cap \unif^{\mu_M} K_M \unif^{-\mu_M})$-conjugation.
  
  Since Zariski locally on $Z$ we have $p_\mu\iv(Z) \cong  Z \times (N(L)\unif^{\mu}K \cap K \unif^\mu K)$, we obtain a bijection of irreducible components as claimed. Note that since $(K_M \cap \unif^{\mu_M} K_M \unif^{-\mu_M})$ is a connected, thus irreducible, group scheme, it acts trivially on the set of irreducible components. Hence the bijection is canonical, as claimed. 
 \end{proof}
 
  \begin{proof}[Proof of Proposition~\ref{prop-intersection-geom}] 
  We note that by precious lemma the fibre over an element $g \in K_M \mu_M K_M$ is isomorphic to $(N(L)\unif^{\mu}K \cap K \unif^\mu K)/K_N$. In particular, it is non-empty if and only if $\mu_M \in S_M(\mu)$, proving the first claim. Moreover, we deduce that for every  $\mu_M \in S_M(\mu)$ and $C_M \in MV_{\mu_M}(\lambda)$ the preimage is equidimensional of dimension
 \begin{align*}
  \dim p_{\lambda,\mu}\iv(C_M) &= \dim C_M + d(\mu,\mu_M) \\
  &= \langle 2\rho_M,\mu_M+\lambda\rangle + d(\mu,\mu_M) \\
  &\leq \langle \rho, \mu+\mu_M \rangle - \langle 2\rho_M,\lambda \rangle \\
  &=\langle \rho,\mu+\lambda \rangle +  \langle \rho - \rho_M, \rangle \\
  &=\langle \rho,\mu+\lambda \rangle = \dim S_\lambda \cap \Flag_{\mu}.
 \end{align*}
 Here the last step follows from the fact that $\mu_M - \lambda$ is a a linear combination of coroots of $M$, as otherwise $S_\lambda^M \cap \Flag_{M,\mu_M} = \emptyset$. Since equality holds if and only if $V_{\mu_M}^M$ is a summand of $\restr{V_\mu}{M^\vee}$, the closure of $p_{\lambda,\mu}\iv(C)$ is a union of topdimensional irreducible components precisely in this case. By the second part of Lemma~\ref{lemma-fibration} applied to $Z = U_M(L)\unif^\lambda K_M \cap K:M \unif^{\mu_M} K_M$ with $\mu_M$ as above, the proposition follows.
 \end{proof} 
 
\subsection{} In order to study the geometry of $\pi^{-1}(mK_M)$, we introduce the following notation. We fix $\delta_N \coloneqq \unif^{2\rho_N^\vee}$, which is an $N$-dominant central element of $M$ and define for $m\in \NN_0, i \in \ZZ$ the subgroup $N(i)_m \coloneqq \delta_N^i K_{N,m} \delta_N^{-i}$. This gives us an exhaustive and separated filtration
\[
 \dotsb \supset N(-1) \supset N(0) \supset N(1) \supset \dotsb 
\]
of $N(L)$ by admissible subgroups as well as a separated filtration by normal subgroups
\[
 N(i) = N(i)_0 \supset N(i)_1 \supset \dotsb 
\]
 Moreover, we consider the derived series 
 \[
  N = N[0] \supset N[1] \supset \dots \supset N[r] = \{0\}
 \]
 and let $N\langle 1 \rangle,\ldots,N\langle r\rangle$ denote its (Abelian) subquotients. We define the notion of admissibility, dimension etc.\ of subset of $N(L), N[i](L), N\langle i \rangle(L)$ analogous to $G$. Moreover, we define these notions on the fibre in $N[i]$ over any $n'' \in N\langle i+1 \rangle$ per transport of structure after choosing an equivariant isomorphism with subgroup $N[i+1]$. This does not depend on the choice of this identification as Lemma~\ref{lemma-N} below shows.

\begin{lem} \label{lemma-M}
 Let $m \in M(L)$. Then $\Ad(m)\colon N(L) \to N(L), n \mapsto m n m\iv$ preserves admissible subsets and reduces their dimension by  $\val\det \Ad_{\Lie N}(m) = \ \langle 2\rho_N, \nu_M(b) \rangle$.
\end{lem}
\begin{proof}
 This follows by the same argument as when $G$ is split (\cite[(5.3.3)]{GHKR:DimADL}, see also \cite[Cor.~10.13]{Hamacher:NewtonPELSh}).
\end{proof}

\begin{lem} \label{lemma-N}
 Let $0 \leq i \leq r-1$
 \begin{subenv}
 \item For any $n \in N[i]$ the isomorphism $\Ad_{N[i+1]}(n) \colon N[i+1] \isom N[i+1]$ preserves admissibility and for any admissible subset $Y$ we have $\dim \Ad_n(Y) = \dim Y$. In particular the notion of admissibility and dimension on the fibres of $N[i] \epi N[i+1]$ is well-defined.
 \item Let $X\subset N[i]$ be an admissible subset such that its image in $N\langle i \rangle$ is admissible of dimension $d_1$ and such that the dimension of the fibres of the map $X \to N\langle i \rangle$ are equidimensional of dimension $d_2$ then $\dim X = d_1+d_2$.
 \end{subenv}
\end{lem}
\begin{proof}
 Let $j$ be big enough such that $n \in N[i](j)$. Since $N[i+1](j)_m = N[i](j)_m \cap N[i+1]$ is normal in $N[i](j)$, the map $Ad_{N[i+1]}(n)$ induces an isomorphism  $N[i+1](L)/N(j)_m \isom N[i+1](L)/N(j)_m$, proving the first assertion. To prove the second assertion, we choose an $m \in \NN$ such that $X$ is $N[i]_m$-stable. Denote  by $Y$ the image of $X$ in $N\langle i \rangle$. Then the fibres of $X/N[i]_m \to Y/N\langle i+1 \rangle_m$ have dimension $d_2 - \dim N[i+1]_m$ and thus
 \begin{align*}
  \dim X &= \dim X/N[i]_m + \dim N[i]_m \\ 
  &= \dim  Y/N\langle i \rangle_m + d_2 - \dim N[i+1]_m + \dim N[i]_m \\
  &= d_1 + d_2 - \dim N\langle i \rangle - \dim N[i+1]_m + \dim N[i]_m \\
  &= d_1+d_2.
 \end{align*}
\end{proof}

As a first step to calculate the dimension of $\pi^{-1}(mK_M)$ as in \ref{sect-XP}, we now calculate the fibres of $L_{1,b_m}$. Note that unless $b$ is basic, $L_{1,b_m}$ is no longer a Lang morphism and will in general not be (pro-)\'etale.

Since it simplyfies the proof, we consider the more general setup of
\[
 L_{m_1,m_2}\colon N(L) \to N(L), n \mapsto m_1 n^{-1} m_1\iv \cdot m_2 \sigma(n) m_2\iv.
\] 
Following the proof in \cite{GHKR:DimADL}, we show that its fibre dimension coincides with the fibre dimension of $\Lie N \to \Lie N, X \mapsto \Ad(m_1)(X) - \Ad(m_2)(\sigma(X))$, which equals
\[
 d(m_1,m_2) \coloneqq d(\Lie N, \Ad(m_1\iv m_2) \sigma) + \val\det \Ad_{m_1},
\]
by \cite[Prop.~4.2.2]{GHKR:DimADL}. Here we denote for an $F$-space $(V,\Phi)$
\[
 d(V,\Phi) \coloneqq - \sum_{\lambda < 0} \lambda \cdot V_\lambda. 
\]

\begin{prop}[cf.~{\cite[Prop.~5.3.2]{GHKR:DimADL}}] \label{prop-reldim-Lang}
 The map $L_{m_1,m_2}\colon N(L) \to N(L)$ is surjective. If $Y \subset N(L)$ is an admissible subset then $L_{m_1,m_2}^{-1}(Y)$ is ind-admissible of dimension $\dim X + d(m_1,m_2)$.
\end{prop}
\begin{proof}
 We first make a couple of reduction steps. We choose $i \in \NN$ big enough, such that $\delta_N^j \cdot m_1, \delta_N^j \cdot m_2 \in M(L)_+ \coloneqq \bigcup_{\mu \in X_\ast(T)_{I,{\rm dom}}}$. That is, conjugation with $\delta_N^j m_1$ and $\delta_N^j m_2$ restrict to maps $N(i) \to N(i)$ for all $i$ and hence so does $L_{\delta_N^j \cdot m_1, \delta_N^j \cdot m_2}$. Replacing $m_1,m_2$ by $\delta_N^j \cdot m_1, \delta_N^j \cdot m_2$ replaces $L_{m_1,m_2}$ by $\delta_N^i L_{m_1,m_2} \delta_N^{-i}$, which does not change surjectivity or admissiblity of the preimage and changes the dimension of the preimage and its conjectured dimension both by $j \cdot \val\det_{\Lie N} \delta_N$. Thus we may assume without loss of generality that $m_1,m_2 \in M(L)_+$.
 
 Now let $Y \subset N(L)$ be an admissible subset, we choose $i,m$ such that $Y \subset N(i)$ is $N(i)_m$-stable. Since $L_{m_1,m_2}$ commutes with conjugation by $\delta_N$, we may assume without loss of generality that $i=0$.
 
 We first show ind-admissibility. Assume that $n\in L_{m_1,m_2}\iv (Y) \cap N(0)$. Then 
 \begin{align*}
  L_{m_1,m_2}(nN(0)_m) &\subset  m_1 N(0)_m n\iv m_1\iv m_2 \sigma(n) N(0)_m m_2\iv \\
  &\subset N(0)_m    m_1  n\iv m_1\iv m_2 \sigma(n)  m_2\iv N(0)_m \\
  &= L_{m_1,m_2}(n) N(0)_m = Y,
 \end{align*} 
 in particular $L_{m_1,m_2}\iv (Y) \cap N(0)$ is admissible. This implies that
 \[
  L_{m_1,m_2}\iv (Y) \cap N(-i) = \delta_N^{-i}\, ( L_{m_1,m_2}\iv (\delta_N^i Y \delta_N^{-i}) \cap N(0) )\, \delta_N^i
 \]
 is also admissible and hence $L_{m_1,m_2}\iv(Y)$ is ind-admissible. 
 
 In order to prove the rest of the statement, denote by $L_{m_1,m_2}[i]\colon Y[i] \to Y[i]$ and $L_{m_1,m_2}\langle i\rangle \colon Y\langle i\rangle \to Y\langle i\rangle$ the maps induced by $L_{m_1,m_2}$. The identical proof as above shows that the preimage for admissible subset of $N[i]$ or $N\langle i \rangle$ is ind-admissible.
 
 \begin{claimsub}
 Let $1 \leq i \leq r$.
 \begin{subenv}
  \item $L_{m_1,m_2}\langle i\rangle$ is surjective and there exists a $j \in \NN$ such that $N\langle i \rangle (j) \subset \image \restr{L_{m_1,m_2}\langle i \rangle}{N\langle i \rangle (0)}$.
  \item For every admissible subset $Y \subset N\langle i \rangle$ the dimension of $L_{m_1,m_2}\langle i \rangle\iv (Y)$ equals \[ \dim Y + d(\Lie N\langle i \rangle, \Ad(m_1\iv m_2)\sigma) + \val\det \Ad_{\Lie N\langle i \rangle}(m_1). \]
  We denote this expression by $d\langle i \rangle (m_1,m_2)$.
 \end{subenv}
 \end{claimsub}
 \begin{proof}
  Since $N\langle i \rangle$ is abstractly isomorphic to some power of the additive group $\GG_a^{n_i}$, we can identify $N\langle i \rangle \cong \Lie N\langle i \rangle \cong L^{n_i}$. With respect to this identification, we have $L_{m_1,m_2}(v) = \Ad(m_1)(v) - \Ad(m_2)\sigma(v)$ and thus the claim is proven in \cite[Prop.~4.2.2]{GHKR:DimADL}.
 \end{proof}

 \begin{claimsub}
 let $1 \leq i \leq r$.
 \begin{subenv}
  \item $L_{m_1,m_2}[i]$ is surjective and there exists a $j \in \NN$ such that $N[ i ] (j) \subset \image \restr{L_{m_1,m_2}\langle i \rangle}{N[i](0)}$.
  \item Let $Y \subset N[i](L)$ be an admissible subset such that its image $Y'' \in N\langle i \rangle$ is admissible and such that each of the fibres of $Y \epi Y''$ is admissible of same dimension $d$. Then the dimension of $L_{m_1,m_2}[ i ]\iv (Y)$ equals \[   \dim Y + d(\Lie N[i], \Ad(m_1\iv m_2)\sigma) + \val\det \Ad_{\Lie N\langle i \rangle}(m_1). \]
  We denote this expression by $d[i ] (m_1,m_2)$.
 \end{subenv}
 \end{claimsub}
 \begin{proof}
  We prove the claim by descending induction on $i$. Since $N[r-1] = N\langle r \rangle$, we have proven the induction beginning in the previous claim.
  
  Note that the root group decomposition gives us a lift $N\langle i+1 \rangle \to N[i]$ in the category of ind-schemes. Hence multiplication defines an isomorphism $N[i] = N\langle i+1\rangle \cdot N[i+1]$. Now for $n' \in N[i+1]$ and $n'' \in N\langle i+1 \rangle$ we have
  \begin{align*}
   L_{m_1,m_2}(n' n'') &= m_1 {n''}\iv {n'}\iv m_1\iv m_2 \sigma(n') \sigma(n'') m_2\iv \\
   &= m_1 n'' m_1\iv L_{m_1,m_2}(n') m_2 \sigma(n'') m_2\iv \\
   &= Ad(m_1 n'' m_1\iv)(L_{m_1,m_2}(n'))\cdot L_{m_1,m_2}(n'').
  \end{align*}
  Thus the first assertion follows directly by the previous claim and induction hypothesis since $N[i](j) = N[i+1](j) \cdot N\langle i+1\rangle(j)$. The image of $L_{m_1,m_2}\iv (Y)$ in $N\langle i \rangle$ has dimension $\dim Y'' +  d\langle i+1 \rangle(m_1,m_2)$ by the previous claim. Moreover, by Lemma~\ref{lemma-N}~(1) and the induction assumption each fibre of $L_{m_1,m_2}\iv (Y) \epi L_{m_1,m_2}\iv (Y'')$ is has dimension $d +  d[i](m_1,m_2)$. Thus by Lemma~\ref{lemma-N}~(2), we obtain the claimed dimension formula.
 \end{proof}
 
 Note that if $i=0$ this finishes the proof if $Y$ has the form as in the claim. In particular if $Y = N(0)$ we get $ \dim N(-j)\cap L_{m_1,m_2}\iv(N(0)) = d(m_1,m_2)$ for $j \gg 0$. Conjugating by $\delta_N^j$, we obtain $ \dim N(0) \cap L_{m_1,m_2}\iv(N(j)) = d(m_1,m_2) + \dim N(j)$. Let $m \gg 0$ such that $N(0)_m \subset N(j)$. Now $L_{m_1,m_2}$ defines a morphism of schemes $\Lbar_{m_1,m_2}\colon N(0)/N(0)_m \to N(0)/N(0)_m$. Since all non-empty fibres are translates of each other, they have the same dimension $d$ and we have 
 \begin{align*}
  \dim N(j) - \dim N(0)_m + d &= \dim \Lbar_{m_1,m_2}\iv(N(j)/N(0)_m) \\ &=  \dim N(j) + d(m_1,m_2) - \dim N(0)_m, 
  \end{align*}
  thus proving $d = d(m_1,m_2)$. In particular the theorem holds true for any $Y \subset \image \restr{L_{m_1,m_2}}{N(0)}$. But for any admissible $Y$ its conjugate by a high enough power of $\delta_N$ is contained in $\image \restr{L_{m_1,m_2}}{N(0)}$ by the second claim, finishing the proof.
 \end{proof}

 We can finally prove the main results.

 \begin{proof}[Proof of Thm.~\ref{thm-main} and Prop.~\ref{prop-ic}]
  By the previous propositions, we have for any $mK_M \in X_{\mu_M}^M(b)$
  \begin{align*}
   \dim \pi\iv (mK_M) &\stackrel{Prop.~\ref{prop-reldim-Lang}}{=} \dim_{N(0)} K\unif^\mu Kb_m\iv \cap N(L) \\
   &\stackrel{Lemma~\ref{lemma-M}}{=} d(\mu,\mu_M) - 2\langle \rho_N, \nu \rangle
   \shortintertext{ and thus }
   \dim \pi\iv(X_{\mu_M}^M(b)) &= d_M(\mu,b) + d(\mu,\mu_M)  - 2\langle \rho_N, \nu \rangle.
  \end{align*}
  The same calculations as in the split case (\cite[Prop.~5.6.1, Thm.~5.8.1]{GHKR:DimADL}) show that if we vary $\mu_M \in I_{\mu,b}$, the integer $d_G(\mu,b)$ is a sharp upper bound for the right hand side with equality if and only if $V_{\mu_M}$ occurs in $\restr{V_\mu}{M^\vee}$. Thus
  \[
   \dim X_\mu(b) = \dim X_\mu^{P \subset G}(b) = \max_{\mu_M \in I_{\mu,b}} \dim \pi\iv(X_{\mu_M}^M(b)) = d_G(\mu,b).
  \]
  
  To prove the second assertion, we note that $[b]_P = N(\L) \cdot [b]_M$ ( \cite[\S~3.6]{Kottwitz:GIsoc2}). We denote 
  \[
   \Irr'([b] \cap K\unif^\mu K) = \{\ell_b(C) \mid C \in \Irr \widetilde{X_\mu(b)}\},
  \]
  analogously for $[b]_P$ and $[b]_M$. While we conjecture that $\Irr'(\cdot ) = \Irr(\cdot )$, we did not prove this claim. By Proposition~\ref{prop-reldim-Lang} and Lemma~\ref{lemma-fibration}, an irreducible component of $[b]_P \cap K\unif^\mu K$ lies in $\Irr'([b]_P \cap K\unif^\mu K)$ if and only if it maps onto an element in $\Irr'([b]_M \cap K\unif^{\mu_M} K)$ where $V_{\mu_M}^M$ occurs in $\restr{V_\mu}{M^\vee}$.  Writing $\Sigma_{\mu,\mu_M} \coloneqq \Irr(N(\L) \unif^{\mu_M} \cap K\unif^\mu K)$, we thus obtain a diagram of bijections
  \begin{center}
   \begin{tikzcd}
    \Irr' ([b]_P \cap K \unif^\mu K) \ar[dashed]{r} \ar{d}{Lem.~\ref{lemma-fibration}} & \bigcup_{\restr{\lambda}{S^\vee} = \lambda([b])} MV_\mu(\lambda) \ar{d}{Prop.~\ref{prop-intersection-geom}} \\
    \bigcup\limits _{V_{\mu_M}^M \subset \restr{V_\mu}{M^\vee}} \Irr'([b]_M \cap K \unif^{\mu_M} K) \ar{r}{Thm.~\ref{thm-superbasic}} \times \Sigma_{\mu,\mu_M}
    & \bigcup\limits_{V_{\mu_M}^M \subset \restr{V_\mu}{M^\vee} \atop \restr{\lambda}{S^\vee} = \lambda([b])}  MV_{\mu_M}(\lambda) \times \Sigma_{\mu,\mu_M},
   \end{tikzcd}
  \end{center}
  where we choose the top map such that the diagram becomes commutative.   yielding a canonical surjection $\Irr X_\mu(b) \to MV_\mu(\lambda([b]))$. Since both sets have the same cardinality by \cite[Thm.~A.3.1]{ZhouZhu:IrredCompADLV}, this proves Thm.~\ref{thm-main}~(2). 
  
 Note that since $ \ell_b(N(\L)\cdot X) = N(\L) \cdot \ell_b(X)$ for any $X \subset M(L)$ by Proposition~\ref{prop-reldim-Lang}, we can deduce Proposition~\ref{prop-reldim-Lang} from the superbasic case and the description of $\Irr'([b]_P \cap K\unif^\mu K)$ above.  
  
 \end{proof}

 Lastly, we record two important special cases for the nature of the morphism $L_{1,b_m}$, which will be used in the following section.

 \begin{lem} \label{lem-Lang-bij}
   If $M = M_b$, then $L_{1,b_m}$ is a universal homeomorphism, or equivalently the induced map $LN(R) \to LN(R)$ is a bijection for all perfect $k$-algebras $R$.
 \end{lem}
 \begin{proof}
  Let $R$ be a perfect $k$-algebra. We first consider the morphism $L_{1,b_m}\langle i \rangle$. We may identify $N\langle i \rangle$ with a finite dimensional $\L$-vector space $V$ and hence the map $L_{1,b_m}$ with $v \mapsto \Phi(v) - v$ where $\Phi$ is the $\sigma$-semilinear map corresponding to $n \mapsto b_m \sigma(n) b_m\iv$. Since $M=M_b$, we know that the Newton slopes of $(V,\Phi)$ are positive. By explicitly calculating on its simple factors (cf.~e.g.\ Beginning of the proof of \cite[Lemma~4.2.1]{GHKR:DimADL}) we confirm that $L_{1,b_m}$ induces a bijection on $LN\langle i \rangle(R)$. Now we can conclude by proving that $L_{1,b_m}[i]$ is a universal homeorphism by backward induction on $i$. If $N[i] = {0}$ this is obvious. If we know that the statement is true for $i+1$, the induction hypothesis follows by applying the $5$-lemma to
  \begin{center}
   \begin{tikzcd}
    1 \ar{r} & LN[i+1](R) \ar{r} \ar{d}{\cong} & LN[i](R) \ar{r} \ar{d}{L_{1,b_m}[i]} & LN\langle i \rangle(R) \ar{r} \ar{d}{\cong} & 1 \\
    1 \ar{r} &  LN[i+1](R) \ar{r} & LN[i](R) \ar{r} & LN\langle i \rangle \ar{r} & 1.
   \end{tikzcd}
  \end{center}

 \end{proof}

\section{Connected components} \label{sect-cc}

  By \cite[Thm.~0.1]{HeZhou:ADLVcomp}, the connected components of $X_{\leq\mu}(b)$ are precisely the the nonempty subsets of the form $X_{\leq\mu}(b)^\omega$ if $b$ is basic and $\mu$ essentially non-central. We derive the general case from this result by following the strategy in the unramified case as in \cite{ChenKisinViehmann:AffDL},\cite{Nie:ADLVIrredComp}. Surely all combinatorial arguments concerning root data carry over, so it remains to generalise the geometric arguments.

  \subsection{} By Lemma~\ref{lemma-reduction-to-adjoint} it suffices to consider the case that $G$ is adjoint and since the conjectured description of $\pi_0(X_\mu(b))$ is compatible with product we also assume that $G$ is simple. We denote by $d$ the number of irreducible components of the Dynkin diagram of $G$ over $\L$. We fix $\mu \in X_\ast(T)_{\Gamma_0,{\rm dom}}$, $[b] \in B(G,\mu)$ and a standard representative $b \in G(\L)$. We denote by $M$ centraliser of its Newton point. 
  In order to reduce to the basic case, we consider the diagram
  \begin{center}
   \begin{tikzcd}
    & \left[b\right]_P \cap \cl{K\unif^\mu K} \ar[two heads]{ld} \ar[hook]{rd} & \\
    \bigcup_{\mu_M \in I_{\leq\mu,b}} \left[b\right]_M \cap \cl{K\unif^{\mu_M} K} & & \left[b\right] \cap\cl{K \unif^\mu K},
    \end{tikzcd}
  \end{center}
 where $I_ {\leq\mu,b} \subset X_\ast(T)_{\I,M-\dom}$ is chosen such that above diagram is correct and the intersections $[b]_M \cap \cl{K\unif^{\mu_M} K}$ are non-empty for $\mu_M \in I_{\leq\mu,b}$. To determine the set $I_{\leq\mu,b}$ we record the following lemma.
 
 \begin{lem} \label{lem-full-intersection}
   Let $\mu_M \in X_\ast(T)_{\I,M-\dom}$. The set $N(L)\unif^{\mu_M} \cap \cl{K \unif^\mu K}$ is non-empty if and only if $\mu_M \leq \mu$; in this case $\unif^{\mu_M}$ is contained each of its irreducible components. In particular, $N(L)\unif^{\mu_M} \cap \cl{K \unif^\mu K}$ is connected.
 \end{lem}
 \begin{proof}
  The first claim follows directly from Proposition~\ref{prop-intersection}. The second claim follows by the usual argument (see for example \cite[Proof of Lem.~4.2]{NgoPolo:CasselmanShalika}).
 \end{proof} 
   
 Hence $\mu_M \in I_{\leq \mu,b}$ if and only if $\mu_M \leq \mu$ and the pair $([b]_M,\mu_M)$ satisfies the generalised Mazur inequality (\ref{eq-Mazur}). We denote by $I_{\leq\mu,b}^{\min}$ the set of minimal elements in $I_{\leq\mu,b}$. 
 
 \begin{prop}[{\cite[\S~2.4, Lem.~6.8]{Nie:ADLVIrredComp}}] \label{prop-nie-6.8}
  Let $\mu$ and $b$ as above.
  \begin{subenv}
   \item The elements in $I_{\leq\mu,b}^{\min}$ are precisely the minuscule coweights in $X_\ast(T)_{M-\dom}$ corresponding to the elements $\kappa \in \pi_1(M)_\I$ satisfying
  \begin{assertionlist}
   \item $\kappa([b]) \equiv \kappa \textnormal{ in } \pi_1(M)_\Gamma$ and
   \item The difference $\mu - \kappa$ in $\pi_1(M)_\Gamma$ is a linear combination of (the images of) $\Sigma^+ \setminus \Sigma_M^+$.
  \end{assertionlist}
   \item The elements in $I_{\leq\mu,b}^{\min}$ are related in the following way. For $\alpha \in \Sigma\setminus \Sigma_M$ and $r \in \NN$ we write $\mu_M \xrightarrow{(\alpha,r)} \mu_M'$ if $\mu_M' - \mu_M = \alpha^\vee - \sigma^i(\alpha^\vee)$ and $\mu_M +\alpha^\vee, \mu_M- \sigma^i(\alpha^\vee) \leq \mu$.
    Then for any $\mu_M,\mu_M' \in I_{\leq\mu,b}^{\min}$ there exists a sequence $\mu_M = x_0 \xrightarrow{(\alpha_0,r_0)} x_1 \xrightarrow{(\alpha_1,r_1)} \dotsb \xrightarrow{(\alpha_{m-1},r_{m-1})} x_r = \mu_M'$ in $I_{\leq\mu,b}^{\min}$. Moreover the sequence can be chosen such that $(\alpha_j,r_j)$ have the following properties.
  \begin{assertionlist}
   \item $\alpha_j^\vee$ is $M$-anti-dominant and $M$-minuscule
   \item The cardinality of the $\sigma$-orbit of $\alpha_j$ is either $d,2d$ or $3d$. If it equals $d$ or $2d$, we can choose $r_j \leq d$ and $r_j \leq 2d$ otherwise.
  \end{assertionlist}
  \end{subenv}
 \end{prop}
   
 \subsection{} As an intermediate step, we are going to prove that $[b] \cap \cl{K\unif^\mu K}$ is connected. Recently this was done by Zhou (\cite[Thm.~A.1.4]{vanHoften:ModpPtsShVar}). Since parts of the proof will be needed for the next step, we provide a sketch of the proof.
 
 Note that by \cite[Thm.~0.1]{HeZhou:ADLVcomp}, $[b]_M \cap \cl{K_M \unif^{\mu_M} K_M}$ is connected for all $\mu_M \in I_{\mu,b}$. Moreover, its fibres in $[b]_P \cap \cl{K \unif^\mu K}$ are isomorphic to $N(L)\unif^{\mu_M} \cap \cl{K \unif^\mu K}$ and thus connected by Lemma~\ref{lem-full-intersection}. Altogether, we obtain
 \[
  \pi_0([b]_P \cap \cl{K\unif^\mu K}) = I_{\leq\mu,M}^{\min}.
 \]
 For the next step, we require some geometric tools.

\begin{lem} \label{lem-cont-to-proj-line}
 The morphism $\AA^1 \cong U_{\alpha,-1} / U_{\alpha,0} \mono G(L)/K$ has a unique continuation $\PP^1 \to G(L)/K$, mapping $\infty$ to $\unif^{-\alpha}K$.
\end{lem}
\begin{proof}
 First assume that $\alpha$ correponds to a unique root $\beta\in\Phi$. Recall that we have an isomorphism $u_\beta\colon \Res_{\tilde{L}/\L} \GG_a \isom U_\beta$ for a certain field extension $\Ltilde/L$ (see for example \cite[Not.~4.8]{Landvogt:Compactification}), which induces
 \[
  \AA^1 \isom U_{\alpha,-1}'/U_{\alpha,0}', z \mapsto u_\alpha(z\cdot \unif_{\Ltilde}\iv)\cdot U_{\alpha,0}'.
 \]
 By \cite[4.11]{Landvogt:Compactification} we have
 \[
  u_\beta(z\cdot \unif_{\Ltilde}\iv)K = u_{-\beta}(z\iv\cdot \unif_{\Ltilde}) \cdot \unif^{-\beta^\vee} K \to_{z \to \infty} \unif^{-\beta^\vee}K.
 \]
 In the case that $\alpha$ corresponds to two (proportional) roots $\beta,2\beta \in \Phi$. For a quadratic field extension $\Ltilde'/\Ltilde$ we denote
 \begin{align*}
  H(\Ltilde',\Ltilde) &\coloneqq \{ (c,d) \in \Res_{\Ltilde'/\Ltilde} \AA^2 \mid \mathrm{N}_{\Ltilde'/\Ltilde}(c) = \trace_{\Ltilde'/\Ltilde}(d) \} \\
  H_0(\Ltilde',\Ltilde) &\coloneqq \{ (c,d) \in H(\Ltilde',\Ltilde) \mid c=0 \}.
 \end{align*}
 Following \cite[Not.~4.14]{Landvogt:Compactification}, there exists $\Ltilde'/\Ltilde$ and a commutative diagram
 \begin{center}
  \begin{tikzcd}
   \Res_{\Ltilde/\L} H_0(\Ltilde',\Ltilde) \ar[hook]{d} \ar{r}{\sim}[swap]{u_{2\beta}} & U_{2\beta} \ar[hook]{d}  \\
   \Res_{\Ltilde/\L} H(\Ltilde',\Ltilde) \ar{r}{\sim}[swap]{u_\beta} & U_\beta
  \end{tikzcd}
 \end{center}
 inducing the isomorphism
 \[
  \AA^1 \isom U_{\alpha,-1}'/U_{\alpha,0}', z \mapsto u_\beta(0,z\cdot \unif_{\Ltilde}\iv)\cdot U_{\alpha,0}'.
 \]
 Similar as in the first case, we deduce from \cite[4.21]{Landvogt:Compactification}
 \[
  u_\beta(0,z\cdot \unif_{\Ltilde}\iv)K = u_{-\beta}(0,z\iv\cdot \unif_{\Ltilde}) \cdot \unif^{-\alpha^\vee} K \to_{z \to \infty} \unif^{-\alpha^\vee}K.
 \]
\end{proof}

\begin{lem} \label{lem-Nie-6.13}
 Let $\alpha \in \Sigma$ and $\lambda \in X_\ast(T)_\I$. Then
 \begin{align*}
   (U_{\alpha,-1}' \setminus U_{\alpha,0}') \cdot \unif^\lambda &\subset
   \begin{cases} 
    K \unif^\lambda K & \textnormal{ if } \langle \alpha, \lambda \rangle < 0 \\   
    K \unif^{\lambda + \alpha^\vee} K & \textnormal{ if } \langle \alpha, \lambda \rangle \geq 0
  \end{cases}
  \shortintertext{ and }
    \unif^\lambda \cdot (U_{\alpha,-1}' \setminus U_{\alpha,0}') &\subset
   \begin{cases} 
    K \unif^{\lambda-\alpha^\vee} K & \textnormal{ if } \langle \alpha, \lambda \rangle \leq 0 \\   
    K \unif^{\lambda} K & \textnormal{ if } \langle \alpha, \lambda \rangle > 0.
  \end{cases}
 \end{align*}
\end{lem}
\begin{proof}
 The proof is identical to the unramified case (\cite[Lem.~6.13]{Nie:ADLVcomp}), applying \cite[4.11,4.21]{Landvogt:Compactification} to the root group calculations.
\end{proof}

\begin{lem} \label{lem-Nie-6.14}
 For any $\lambda \in X_\ast(T)_\I$ and $\alpha,\beta \in \Sigma$ such that the root system generated by $\alpha,\beta$ is of type $A$ and $\lambda,\lambda+\alpha^\vee,\lambda-\beta^\vee, \lambda + \alpha^\vee - \beta^\vee \succeq \mu$ we have
 \[
  U_{\alpha,-1}' \unif^\lambda U_{\beta,-1} \subset \cl{K \unif^\mu K}.
 \]
\end{lem}
\begin{proof}
 This is can be concluded from the previous lemma by the same proof as in the unramified case (\cite[Lem.~6.14]{Nie:ADLVcomp}).
\end{proof}

\begin{thm}[{\cite[Thm.~A.1.4,Lem.~A.3.11]{vanHoften:ModpPtsShVar}}] \label{thm-nie-6.8} 
 $[b] \cap  \cl{K \unif^{\mu} K}$ is connected, i.e.\ $J_b(F)$ acts transitively on $\pi_0(X_\mu(b))$. More precisely, if $\mu_M,\mu_M' \in I_{\leq\mu,b}^{\min}$ with $\mu_M \xrightarrow{(\alpha,r)} \mu_M'$ and $\omega_M \in \pi_1(M)$, then $X_{\mu_M}^M(b)^{\omega_M}$ lies in the same connected component of $X_\mu(b)$ as $X_{\mu_M'}(b)^{\omega_M + \sum_{i=0}^{r-1} \sigma^i(\alpha)}$.
\end{thm}
\begin{proof}[Sketch of proof]
 As a consequence of the Iwasawa decomposition, we have a surjective morphism
 \[
  K \times ([b]_P \cap \cl{K \unif^\mu K}) \to [b] \cap \cl{K \unif^\mu K}, (g,b') \mapsto g\iv b' \sigma(g).
 \]
 Since $K$ is connected it thus suffices to prove that the image of $[b]_P \cap \cl{K \unif^\mu K}$ lies in the same connected component of $[b] \cap \cl{K \unif^\mu K}$. By Propostion~\ref{prop-nie-6.8}, it suffices to connect the components of $[b]_P \cap \cl{K \unif^\mu K}$ corresponding to a pair $\mu_M \xrightarrow{(\alpha,r)} \mu_M'$ with $(\alpha,r)$ satisifying properties (a) and (b) of the proposition.
 Let $g \in X_{\mu_M}(b)^{\omega_M}$. By an explicit calculation (see \cite[Proof of Lem.~A.3.11]{vanHoften:ModpPtsShVar}) the morphism
 \[
  \phi\colon \AA^1 \cong U_{\alpha,-1}/U_{\alpha,0} \to \Flag_G, u \mapsto u\sigma(u) \dotsm \sigma^{r-1}(u)
 \]
 factors through $X_{\leq \mu}(b)$. Its continuation to $\PP^1$ satisfies $\phi(\infty) = \unif^{\alpha + \dotsb + \sigma^{r-1}(\alpha)}$ by Lemma~\ref{lem-cont-to-proj-line}, finishing the proof.

\end{proof}


 \subsection{} We assume from now on that $(b,\mu)$ is HN-irreducible. We fix a connected component $X_0 \subset X_\mu(b)$ and denote by $H_b \subset J_b(F)$ the stabiliser of $X_0$. To prove Theorem~\ref{thm-main}~(3), it suffices to show that $H_b = J_b(F)^0$. As we already know by \cite{ZhouZhu:IrredCompADLV} that $J_b(F)^{0_M} \subset H_b$ it even suffices to check that a set of representatives of $J_b(F)^0/J_b(F)^{0_M}$ is contained in $H_b$. Now
 \[
  J_b(F)^0/J_b(F)^{0,M} \cong \ker(\pi_1(M)_\I \epi \pi_1(G)_\I)^\sigma \cong (\ZZ\cdot \Sigma / \ZZ\cdot \Sigma_M)^\sigma.
 \]
 In particular, as the quotient is Abelian, it follows that $H_b$ is a normal such group of $J_b(F)$ is hence does not depend on the choice of $X_0$. It remains to show that $H_b$ contains a set of generators of $\ZZ\cdot \Sigma / \ZZ\cdot \Sigma_M)^\sigma$. We fix $\mu_M$ such that $X^0$ contains a connected component of $X_{\mu_M}^M(b)$.
 
 For this let
 \begin{align*}
  C &\coloneqq \{ \alpha \in \Sigma^+ \mid \mu_M + \alpha \leq \lambda, \alpha\ M\textnormal{-minuscule and }M\textnormal{-antidom.} \}
  \shortintertext{unless $\Sigma$ is of type $G_2$ and $\Sigma_M \subset \Sigma_G$ are the short roots, in which case let}
  C &\coloneqq \{ \alpha \in \Sigma^+_{long} \mid \mu_M + \alpha \leq \lambda, M\textnormal{-antidom.}\}
 \end{align*}
  By \cite[Lemma~7.1,7.2]{Nie:ADLVcomp} the set $\underline{C} \coloneqq \{\underline\alpha \mid \alpha \in C\}$ generates $(\ZZ\cdot \Sigma / \ZZ\cdot \Sigma_M)^\sigma$ where $\underline\alpha$ denotes the sum over all elements of the $\sigma$-orbit of $\alpha$.
 
 \begin{proof}[Proof of Theorem~\ref{thm-main}~(3)]
  It remains to show that for every $\alpha \in C$ we have $\unif^{\underline\alpha} \in H_b$. We denote by $d$ the number of connected components of the Dynkin diagram of $\Sigma$ and let $n \coloneqq \# \sigma^\ZZ.\alpha$. Then $n \in \{d,2d,3d\}$ By \cite[Lemma~8.1]{Nie:ADLVcomp} either
  \begin{assertionlist}
   \item There exist $\theta \in \sigma^\ZZ.\alpha, 1\leq j \leq n-1$ and $\mu_M' \in I_{\leq\mu,b}^{\min}$ such that $\mu_M \xrightarrow{(\theta,j)} \mu_M'$ or
   \item For all $1\leq j \leq n-1$ with $d \not| j$ we have $w_b(\sigma^k(\alpha) = \sigma^k(\alpha)$ and $\langle \sigma^k(\alpha),\mu_M \rangle = 0$.
  \end{assertionlist}  
  
  \smallskip
  
  \emph{Case $n=d$:} If (a) holds then we have $\mu_M \xrightarrow{(\theta,j)} \mu_M' \xrightarrow{(\sigma^k(\theta),d-j)} \mu_M $, thus $\unif^{\underline\alpha} \in H_b$ by Theorem~\ref{thm-nie-6.8}. If (b) holds, we have
   \[
    u_\alpha \Ad_{b\sigma}(u_\alpha) \cdots \Ad_{b\sigma}^{n-1}(u_\alpha) K \in X_\mu(b)
   \]
   for $u_\alpha \in U_{\alpha,-1}$ by explicit calculations on root groups (cf.~\cite[8.2,Case 2]{Nie:ADLVcomp}). Since $X_\mu(b)$ is closed in $G(L)/K$, we deduce by Lemma~\ref{lem-cont-to-proj-line} that $\underline\alpha = \alpha + \dotsb + (w_b\sigma)^{n-1}(\alpha) \in H_b$.
   
  \smallskip   
 
  \emph{Case $n=2d$:} If (a) holds, then $\unif^{\underline\alpha} \in H_b$ by the same argument as above. Otherwise, $\alpha' \coloneqq (w_b\sigma)^d(\alpha) + \alpha \in \Sigma$ (cf.~\cite[Lemma ~8.3]{Nie:ADLVcomp}) and we define a morphism $\phi\colon\AA^{1,\perf} \cong U_{\alpha,-1}/U_{\alpha,0} \to X_\mu(b)$ as follows. If $w_b\sigma^d(\alpha) \not= \sigma^d(\alpha)$ or $\langle \alpha, \mu_M \rangle \geq 0$, we denote
  \[
   \phi(u_\alpha) \coloneqq u_\alpha \cdot \Ad_{b\sigma}(u_\alpha) \cdots \Ad_{b\sigma}^{d-1}(u_\alpha) \cdot u_{\alpha'} \cdots \Ad_{b\sigma}^{d-1}(u_\alpha')),
  \]
  where $u_{\alpha'} \in U'_{\alpha',-1}$ is defined by
  \[
   u_{\alpha')} \coloneqq [u_\alpha, \unif^{\mu_M}Ad_{b\sigma}^d(u_\alpha)\unif^{-\mu_M}].
  \]
  Using the same calculations as in the unramified case(\cite[4.2.14]{ChenKisinViehmann:AffDL}, \cite[p.~1392f]{Nie:ADLVcomp}), this morphism indeed maps into $X_\mu(b)$ and can be continued by Lemma~\ref{lem-cont-to-proj-line} to a morphism on $\PP^{1,\perf}$, mapping $\infty$ to $\unif^{\underline\alpha}$. Thus $\underline\alpha \in H_b$. By \cite[Lemma ~8.3]{Nie:ADLVcomp}, the only remaining case is $w_b\sigma^d(\alpha) = \sigma^d(\alpha)$ and $\langle \alpha, \mu_M \rangle = -1$. This case follows by the same construction as the unramified case (\cite[4.6.16]{ChenKisinViehmann:AffDL}), which while being more complicated than the previous case, only relies on root theoretic calculations and Lemma~\ref{lem-cont-to-proj-line}.
  
  \smallskip
  
  \emph{Case $n=3d$:} Note that in case (a), we cannot automatically apply Theorem~\ref{thm-nie-6.8}, as this requires that $j,n-j < 2d$. In the remaining cases it suffices by symmetry  to consider $\mu_M \xrightarrow{(\theta,j)} \mu_M' \xrightarrow{(\sigma^j(\theta),n-j)} \mu_M$ with $j \geq 2d$. It now follows from a combinatorial calculation with $D_4$-root systems (\cite[Lemma~8.5 and p.~1396]{Nie:ADLVcomp}) that for any $u_\theta \in U_{\theta,-1}$ we have
  \[
   \Ad^{j-1}(b\sigma)(u_\theta)\cdots u_\theta \cdot K \in X_\mu(b)
  \]
  and thus we conclude by Lemma~\ref{lem-cont-to-proj-line} that $X_{\mu_M}(b)^{0}$ and $X_{\mu_M}(b)^{\alpha + \dotsb + \sigma^{j-1}(\alpha)}$ lie in the same connected component of $X_\mu(b)$. By applying Theorem~\ref{thm-nie-6.8} to $\mu_M' \xrightarrow{(\sigma^j(\theta),n-j)} \mu_M$, we can conclude that $\unif^{\underline\alpha} \in H_b$.
  
  In case (b), we have
  \[
   \Ad_{b\sigma}^{n-1}(u_\alpha) \cdots u_\alpha \in X_\mu(b)
  \]
  by the same calculations as in the unramified case (\cite{ChenKisinViehmann:AffDL}, \cite[p.1395]{Nie:ADLVcomp}). We conclude by Lemma~\ref{lem-cont-to-proj-line} that $\unif^{\underline\alpha} \in H_b$.
  
 \end{proof}
 

\section{Application to Shimura varieties} \label{sect-Shimura}

\subsection{} \label{ssect-KP-deformation-space}
 Assuming the following conditions are satisfied, Kisin and Pappas  associated a deformation ring $R_G$ to the datum $(G,b,\mu,\Kscr)$ in \cite[\S~3.2]{KisinPappas:ParahoricIntModel}. We assume that $F = \QQ_p$ and that there exists an embedding $\Kscr \mono GL_{n,\ZZ_p}$ such that the following condition are satisfied.
 \begin{assertionlist}
  \item The image of $\Kscr$ in $\GL_{n,\ZZ_p}$ contains the scalar matrices.
  \item The image of $\mu$ is conjugated to the cocharacter $\GG_m \to GL_n, t \mapsto (1,\dotsc,1,t\iv,\dotsc,t\iv)$
  \item There exists a p-divisible group $X$ over $\k$  with Dieudonn\'e crystal $(\ZZ_p^n, b\sigma)$. Furthermore, there exists a cocharacter $\mu\colon \GG_m \to \GL_{n,O_K}$ for a finite extension $K/\L$ such that its generic fibre factors through $G_K$ and induces the Hodge filtration of $X$.
 \end{assertionlist}
 We choose a family $(s_\alpha)$ of elements in $(\ZZ_p^n)^\otimes$ such that $\Kscr$ is their stabiliser. By construction, $\Spf R_G$ is a closed subspace of the deformation space of $X$ and the restriction $X_G^\wedge$ of the universal deformation of $X$ to $\Spf R_G$ is canonically equipped with crystalline Tate-tensors $(t_\alpha^\wedge)$, which specialise to $(s_\alpha)$ over $\Spec k$.
 
\subsection{}
 We denote by $(X_G, t_\alpha)$ the algebraisation of $X_G^\wedge, (t_\alpha^\wedge)$ over $\Def_G \coloneqq \Spec R_G$. Then the isogeny classes of $(X_G,t_\alpha)$ over geometric points are parametrised by a subset of $B(G)$. We denote the Newton stratum corresponding to $\bbf \in B(G)$ by $\Def_G^\bbf$. By \cite[Thm.~1.1]{Kim:CenLeaf} the dimension of $\Def_G^{[b]}$ equals $\dim_{eK} X_\mu(b) + \langle 2\rho,\nu([b]) \rangle$, where $e \in G(L)$  denotes the unit. Since $\dim_{eK} X_\mu(b) \leq d(b,\mu)$ by Theorem~\ref{thm-main}, we obtain an estimate
 \begin{equation} \label{eq-dimension-bound-1}
  \dim \Def_G^{[b]} \leq \langle \rho, \mu+\nu([b] \rangle - \frac{1}{2} \defect_G(b).
 \end{equation}
 This estimate can be expressed in more natural terms. For two elements $\bbf,\bbf' \in B(G)$, we write $\bbf' \leq \bbf$ if $ \nu(\bbf') \leq \nu(\bbf)$ and $\kappa(\bbf') = \kappa(\bbf)$. For $\bbf' \leq \bbf$ we denote by $\ell(\bbf',\bbf)$ the maximal length of a chain $\bbf' < b_1 < \dotsb < \bbf$ in $B[G)$. Using the corrected version of Chai's formula \cite[Thm.~7.4]{Chai:NewtonPolygons} (see also \cite[Prop.~3.11]{Hamacher:NewtonPELSh}, \cite[Thm.~3.4]{Viehmann:NewtonStrat}) the inequality (\ref{eq-dimension-bound-1}) is equivalent to
 \begin{equation} \label{eq-dimension-bound-2}
  \codim \Def_G^{[b]} \geq \ell([b],\bbf_{max}),
 \end{equation}
 where $\bbf_{max} = [\unif^\mu]$ denotes the unique maximal element of $B(G,\mu)$ (cf.~\cite[Thm.~1.1]{HeNie:AcceptableElements}).
 
 We can now show that the above inequalities are in fact equalities.
 
 \begin{prop} \label{prop-Newton-Def}
  For any $[b] \leq \bbf \leq \bbf_{max}$ we have
  \begin{subenv}
   \item $\Def_G^{\bbf}$ is of pure codimension $\ell([b],\bbf_{max})$ in $\Def_G$
   \item $\cl{\Def_G^{\bbf}} = \bigcup_{\bbf' \leq \bbf} \Def_G^{\bbf'}$
  \end{subenv}
 \end{prop} 
 \begin{proof}
  This follows by a formal argument from (\ref{eq-dimension-bound-2}) as in \cite[Prop.~3.10]{Viehmann:NewtonStrat}.
 \end{proof}
 
  \begin{cor} \label{cor-equidim}
   Assume that $F = \QQ_p$ and that $(G,\mu,b,\Kscr)$ satisfy the conditions of \S~\ref{ssect-KP-deformation-space}. Then $X_\mu(b)$ is equidimensional. 
  \end{cor}
  \begin{proof}
   We consider an arbitrary point $gK \in X_\mu(b)$. Let $b' \coloneqq g\iv b\sigma(g)$ such that we have an isomorphism $X_\mu(b') \isom X_\mu(b), x \mapsto g\cdot x$. Note that $(G,\mu,b',\Kscr)$ also satisfies the conditions of \S~\ref{ssect-KP-deformation-space}. Thus by the previous proposition the inequality (\ref{eq-dimension-bound-1}) is an equality for $b'$ and hence
   \[
    \dim_{gK} X_\mu(b) = \dim_{eK} X_{\mu}(b') = d(\mu,b') = d(\mu,b).
   \]
  \end{proof}

 \subsection{} 
 Let $\Sscr_{\Gsf,K,0}$ denote the special fibre of the Kisin-Pappas model of a Shimura variety of Hodge type with very special level $K$ at $p$.  We denote by $ \Sscr_{\Gsf,K,0}^{\bbf}$ the Newton stratum associated to an element $\bbf \in B(\Gsf_{\QQ_p},\mu)$.
 
 \begin{cor} \label{cor-Newton-Sh}
  For any $\bbf \in B(\Gsf_{\QQ_p},\mu)$, we have
  \begin{subenv}
   \item $\Sscr_{\Gsf,K,0}^{\bbf}$ is of pure codimension $\ell([b],\bbf_{max})$ in $\Sscr_{\Gsf,K,0}$.
   \item $\cl{\Sscr_{\Gsf,K,0}^{\bbf}} = \bigcup_{\bbf' \leq \bbf} \Sscr_{\Gsf,K,0}^{\bbf}$.
  \end{subenv}
 \end{cor}
 \begin{proof}
  Let $x \in \Sscr_{\Gsf,K,0}^{\bbf}(\k)$. Then  $ \Spec \Oscr_{\Sscr_{\Gsf,K,0},x}^\wedge$ is isomorphic to a deformation space $\Def_G$ as above (\cite[Cor.~4.2.4]{KisinPappas:ParahoricIntModel}) and this isomorphism preserves Newton strata (\cite[Prop.~4.6]{HamacherKim:Mantovan}). Thus the claim is a direct consequence of Proposition~\ref{prop-Newton-Def}.
 \end{proof}

\subsection{}
  New assume that $\cha F = p$. We denote fix $b \in G(L)$ and denote by $\Def_{G,\mu}$ the deformation space of the local shtuka $(K,b\sigma)$ bounded by $\mu$. With the analogous notation as above we obtain the following results.
  
  \begin{prop} \label{prop-Newton-Def-Sht}
  For any $[b] \leq \bbf \leq \bbf_{max}$ we have
  \begin{subenv}
   \item $\Def_{G,\mu}^{\bbf}$ is of pure codimension $\ell([b],\bbf_{max})$ in $\Def_G$
   \item $\cl{\Def_{G,\mu}^{\bbf}} = \bigcup_{\bbf' \leq \bbf} \Def_G^{\bbf',\mu}$
  \end{subenv}
 \end{prop} 
 \begin{proof}
  By \cite[Lemma~3.9]{MilicevicViehmann:GenericNP}, we have the dimension of $ Def_{G,\mu}^{[b]}$ equals $ \dim_{eK} X_{\leq\mu}(b) + \langle 2\rho,\nu([b])\rangle$. Since $\dim_{eK} X_{\leq\mu}(b) \leq d(b,\mu)$ by Theorem~\ref{thm-main} and $\dim \Def_{G,\mu} = \langle 2\rho,\mu \rangle$ by \cite[Prop.~3.8]{MilicevicViehmann:GenericNP}, we thus obtain
  \[
   \codim \Def_{G,\mu}^{[b]} \geq \ell([b],\bbf_{max}).
  \]
  Now we conclude by \cite[Prop.~3.10]{Viehmann:NewtonStrat}.
 \end{proof}
 
 \begin{cor} \label{cor-equidim-2}
  Assume that $\cha F = p$. Then $X_{\leq \mu} (b)$ (and thus $X_\mu(b)$) is equidimensional.
 \end{cor}
 \begin{proof}
  The proof is the same as the one of Corollary~\ref{cor-equidim} above.
 \end{proof}
 
 \begin{cor}
  Fix $\mu \in X_\ast(T)_I$. For any $[b] \leq \bbf \leq \bbf_{max}$ we have
  \begin{subenv}
   \item $(\cl{K\unif^\mu K})^{\bbf}$ is of pure codimension $\ell([b],\bbf_{max})$ in $\cl{K\unif^\mu K}$.
   \item The closure of $(\cl{K\unif^\mu K})^\bbf$ equal $\bigcup_{\bbf' \leq \bbf} (\cl{K\unif^\mu K})^{\bbf'}$.
  \end{subenv}
 \end{cor} 
 \begin{proof}
  Choose $n \in \NN$ such that $[b] \cap \cl{K \unif^\mu K}$ is right-$K_n$ stable for all $[b] \in B(G,\mu)$.  By definition, the formal neighbourhood $b \in \cl{K \unif^{\mu} K}/K_n$ is isomorphic to $\Def_{G,\mu,n}$, parametrising the deformations of the local $K_n$-shtuka $(K_n,b\sigma)$. Note that $\Def_{G,\mu,n} \cong (K/K_n)_e^\wedge \times \Def_{G,\mu}$, compatible with Newton stratification (see also \cite[p.~7]{MilicevicViehmann:GenericNP}), the statement follows from Proposition~\ref{prop-Newton-Def-Sht}.
 \end{proof}
 \def\cprime{$'$}
\providecommand{\bysame}{\leavevmode\hbox to3em{\hrulefill}\thinspace}
\providecommand{\MR}{\relax\ifhmode\unskip\space\fi MR }
\providecommand{\MRhref}[2]{%
  \href{http://www.ams.org/mathscinet-getitem?mr=#1}{#2}
}
\providecommand{\href}[2]{#2}

 \end{document}